\documentclass[10pt,reqno,final]{amsart}
\usepackage{epsfig,amssymb,amsmath,version}
\usepackage{amssymb,version,graphicx,fancybox,mathrsfs,multirow}

\usepackage{url,hyperref}
\usepackage[notcite,notref]{showkeys}

\usepackage{subfigure}
\usepackage{color,xcolor}
\usepackage{cases}
\usepackage{mathtools}

\catcode`\@=11 \theoremstyle{plain}
\@addtoreset{equation}{section}   
\renewcommand{\theequation}{\arabic{section}.\arabic{equation}}
\@addtoreset{figure}{section}
\renewcommand\thefigure{\thesection.\@arabic\c@figure}
\renewcommand{\thefigure}{\arabic{section}.\arabic{figure}}
\newtheorem{thm}{\bf Theorem}

\newenvironment{theorem}{\begin{thm}} {\end{thm}}
\newtheorem{cor}{\bf Corollary}

\newenvironment{corollary}{\begin{cor}} {\end{cor}}
\newtheorem{lmm}{\bf Lemma}

\theoremstyle{remark}
\newtheorem{rem}{\bf Remark}[section]
\theoremstyle{definition}

\numberwithin{table}{section}

\def \af {\alpha}
\def \bt {\beta}

\newcommand{\bs}[1]{\boldsymbol{#1}}

\renewcommand \wedge \times

\begin{document}
\bibliographystyle{plain}
\baselineskip 13pt

\title[Fractional Collocation Methods] {Well-Conditioned Fractional   Collocation Methods Using Fractional Birkhoff Interpolation Basis 
}
\author[Y. Jiao, \;  L. Wang \; $\&$ \; C. Huang]{Yujian Jiao${}^1$, \;\; Li-Lian Wang${}^{2*}$ \;\; and \;\; Can Huang${}^3$}

 \thanks{${}^1$Department of Mathematics, Shanghai Normal University, Shanghai 200234, P. R. China, and  Scientific Computing Key
Laboratory of Shanghai Universities. This author is supported in part by NSFC grants (No. 11171227 and No.  11371123), Natural Science Foundation of Shanghai 
(No.13ZR1429800), and the State Scholarship Fund of China  (No. 201308310188). \\
\indent ${}^{2*}$({\it Corresponding author: }lilian@ntu.edu.sg) Division of Mathematical Sciences, School of Physical
and Mathematical Sciences, Nanyang Technological University,
637371, Singapore. The research of this author is partially supported by Singapore MOE AcRF Tier 1 Grant (RG 15/12), Singapore MOE AcRF Tier 2 Grant (MOE 2013-T2-1-095, ARC 44/13) and Singapore A$^\ast$STAR-SERC-PSF Grant (122-PSF-007). \\
\indent ${}^3$School of Mathematical Sciences  and Fujian Provincial Key Laboratory on Mathematical Modeling \& High Performance Scientific Computing, Xiamen University, Fujian 361005, China. The research of this author is supported by National Natural Science Foundation of China under Grant 11401500.  \\
\indent The first and last authors thank the hospitality of the Division of Mathematical Sciences, School of Physical
and Mathematical Sciences, Nanyang Technological University, Singapore,  for hosting their visit. 
}
\begin{abstract} The purpose of this paper is twofold.  Firstly, we provide explicit and compact formulas for computing  both Caputo and (modified)  Riemann-Liouville (RL) fractional pseudospectral differentiation matrices (F-PSDMs) of any order at general Jacobi-Gauss-Lobatto (JGL) points.  
We show that in the Caputo case, it suffices to compute   F-PSDM of order $\mu\in (0,1)$ to compute  that of  any order $k+\mu$ with  integer $k\ge 0,$  while in the modified RL case, it is only  necessary to evaluate a fractional integral matrix of order $\mu\in (0,1).$ 
Secondly,  we introduce suitable fractional JGL  Birkhoff interpolation problems  leading to new  interpolation polynomial basis functions  with remarkable properties: (i)  the matrix generated from the new  basis yields the exact inverse of  F-PSDM at  ``interior" JGL points;  
(ii)  the matrix  of the highest fractional derivative in a collocation scheme under the new basis is diagonal;    and
(iii)  the resulted linear system is well-conditioned in the Caputo case, while in the modified RL case, the eigenvalues of the coefficient matrix are highly concentrated. In both cases,   the linear systems of the collocation schemes using the new basis   can solved by an iterative solver within  a few iterations. 
Notably, the inverse can be computed in a very stable manner,  so  this offers optimal preconditioners   for usual fractional collocation methods for fractional differential equations (FDEs). It is also noteworthy that the choice of certain special JGL points with parameters related to the order of the equations can ease the implementation.            
We highlight that the use of  the  Bateman's fractional integral formulas and fast  transforms between Jacobi polynomials with different parameters, are  essential for our algorithm development.       
\end{abstract}
\keywords{Fractional differential equations, Caputo fractional derivative, (modified) Riemann-Liouville fractional derivative, fractional Birkhoff interpolation, interpolation basis polynomials, well-conditioned collocation methods}
 \subjclass[2000]{65N35, 65E05, 65M70,  41A05, 41A10, 41A25}

\maketitle
 
\vspace*{-12pt} 
\section{Introduction}
 Fractional differential equations have been found   more realistic  in modelling a variety of  
physical phenomena, engineering processes, biological systems and financial products, such as  anomalous diffusion and non-exponential relaxation patterns,  viscoelastic   materials and among others. 
Typically, such scenarios involve long-range temporal cumulative memory  effects and/or long-range spatial interactions that 
can be more accurately described by  fractional-order  models (see, e.g.,  \cite{Pod99,Met.K00,Kil.S06,Diet10,DGLZ13b} and the references therein).

One  challenge  in numerical solutions  of FDEs resides  in that the underlying  fractional integral and derivative operators are global in nature. Indeed, it is not surprising to see the
finite difference/finite element methods based on ``local operations" leads to full and dense matrices (cf. \cite{Mee.T04,LAT04,Sun.W06,MST06,EHR07,Erv.R07,Tad.M07,JLZ13} and the references therein), which are expensive to compute  and invert.  It is therefore of importance to 
construct fast solvers by carefully analysing the structures of the matrices  (see, e.g., \cite{Wan.B12,lin2014preconditioned}).  This should be  in marked contrast with the situations when they are  applied to differential equations of integer order derivatives.  
In this aspect, the spectral method using global basis functions  appears to be  well-suited for non-local problems.  However, only limited efforts have been devoted  to  this very promising approach 
 (see, e.g.,  \cite{Li.X09,Li.X10,LZL12,zayernouri2013fractional,XuH14,CSW2014}), when compared with a large volume of literature on finite difference and finite element methods.

Another more distinctive challenge in solving FDEs  lies in that the intrinsic singular kernels  of  the fractional integral and derivative operators induce singular solutions and/or data.
Just to mention a simple FDE involving RL fractional derivatives order $\mu\in (0,1)$: 
$ \big(~\, {^{R}}\hspace*{-12pt}{}_{-1}D_{\!x}^\mu\, u\big)(x)=1$ for $x\in (-1,1),$ such that $u(-1)=0,$ whose  solution behaves like  $u(x)\sim (1+x)^{\mu}.$ Accordingly,  it only has a limited regularity in a usual Sobolev space, so the naive polynomial approximation  has  a poor convergence rate. 
 Zayernouri and Karniadakis \cite{zayernouri2013fractional} proposed to approximate such singular solutions by  Jacobi poly-fractonomials (JPFs), which were derived from  eigenfunctions of a fractional Sturm-Liouville operator. Chen, Shen and Wang \cite{CSW2014}  modified the generalised Jacobi functions (GJFs) introduced earlier in Guo, Shen and Wang \cite{Guo.SW09},  and  rigorously derived the approximation results in weighted Sobolev spaces involving  fractional derivatives.
 The JPFs turned out to be special cases of GJFs, and the GJF Petrov-spectral-Galerkin methods could achieve   truly spectral convergence for some prototypical FDEs. We also refer to \cite{wang2015high} for  interesting attempts to characterise the regularity of solutions to some special FDEs by Besov spaces.   It is also noteworthy that the analysis of  spectral-Galerkin approximation in \cite{Li.X09,Li.X10} was under the function spaces and notion in   \cite{Erv.R07}, and in \cite{JLZ13},  the finite-element method was analyzed for  the case with smooth source term but  singular solution. 

 It is known that by pre-computing the pseudospectral differentiation matrices (PSDMs),  
  the  collocation method enjoys  a  ``plug-and-play" function  with simply replacing derivatives by PSDMs, so it has remarkable advantages in dealing with variable coefficients and nonlinear PDEs . However, the practicers are usually plagued with the dense, ill-conditioned  linear systems, when compared with properly designed spectral-Galerkin approaches  (see, e.g.,  \cite{CHQZ06,ShenTangWang2011}).    The ``local" finite-element  preconditioners  (see, e.g., \cite{Kim.P97}) and  ``global" integration preconditioners   (see, e.g., \cite{coutsias1996integration,Greengard91,Hesthaven98,Elbarbary06,Wan.SZ14,WangZhang2}) were developed to overcome the ill-conditioning of the linear systems.  When it comes to FDEs, it is advantageous to use collocation methods, as the Galerkin approaches usually lead to full dense matrices as well.  
 Recently, the development of collocation methods for FDEs has attracted much attention (see, e.g., \cite{LZL12,zayernouri2014fractional,tian2014polynomial,fatone2014optimal}).  It was numerically testified in \cite{LZL12,zayernouri2014fractional} that for both Lagrange polynomial-based and JPF-based collocation methods, the condition number of the Caputo F-PSDM of order $\mu$ behaves like $O(N^{2\mu})$ which  is consistent with the integer-order  case.  However, it seems very difficult to construct preconditioners from finite difference and finite elements as they own involve full and dense matrices and suffer from ill-conditioning.   

The main purpose of this paper is to construct integration preconditioners and new basis functions for well-conditioned fractional collocation methods  from  some suitably defined  {\em fractional Birkhoff polynomial interpolation problems.}  In \cite{Wan.SZ14},  optimal integration preconditioners 
were  devised for PSDMs of integer order, which allows for stable implementation of collocation schemes even for thousands  of collocation points.  
 Following the spirit of  \cite{Wan.SZ14}, we introduce suitable  fractional Birkhoff interpolation problems at general JGL points  with respect to both Caputo  and (modified) Riemann-Liouville fractional derivatives (note: the RL fractional derivative is  modified by removing the singular factor so that it is  well defined at every collocation point).  As we will see, the extension is nontrivial and much more involved than the integer-order derivative case.  Here, we restrict our attention to the polynomial approximation, though the ideas and techniques can be extended to JPF- and GJF-type basis functions.   On the other hand, using a  suitable mapping, we can transform the FDE (e.g., the aforementioned example) and approximate the smooth solution of the transformed equation, which is alternative to the direct use of JPF or GJF approximation to achieve spectral accuracy for certain special FDEs.   
 
 We highlight the main contributions of this paper  in order. 
 \begin{itemize}
 \item From the fractional Birkhoff interpolation,  we derive  new interpolation basis polynomials with  remarkable properties:
 \begin{itemize}
 \item[(i)]  It provides a stable way to compute the exact inverse of Caputo and (modified) Riemann-Liouville fractional PSDMs associated with ``interior" JGL points.  This offers integral preconditioners for fractional collocation schemes using Lagrange interpolation basis polynomials. 
 \item[(ii)] Using the new basis, the matrix of the highest fractional derivative in a collocation scheme is identity,  and the F-PSDMs are not involved.   More importantly, 
 the resulted linear systems can be solved by an iterative method converging within a few iterations even for a very large number of collocation points. 
 \end{itemize}
 \item We propose a compact and systematic way to compute  Caputo and (modified) Riemann-Liouville F-PSDMs of any order at JGL points. In fact,  we can show that  the computation of F-PSDM of order $k+\mu$ with $k\in {\mathbb N}$ and $\mu\in (0,1)$  boils down to evaluating  (i)   F-PSDM of order $\mu$  in the Caputo case, and (ii) a modified fractional integral matrix of order $\mu$  in the Riemann-Liouville case.     
  Using the Bateman's fractional integral formulas and the connection problem, i.e.,  the transform between Jacobi polynomials with different parameters, we obtain the explicit formulas of these matrices. 
 \end{itemize}

The rest of the  paper is organised as follows. The next section is for some preparations.
 In Section \ref{sect:FPSDM}, we present algorithms for computing  Caputo and (modified) Riemann-Louville F-PSDMs.   In Sections \ref{sect:caputoinverse}-\ref{sect:mRL}, we introduce fractional Birkhoff polynomial interpolation and compute new basis functions. Then we are able to stably compute the inverse of F-PSDMs at ``interior" JGL points and construct well-conditioned collocation schemes.  The final section is for numerical results and concluding remarks.

\section{Preliminaries}\label{sect:prem}
In this section, we make necessary preparations for subsequent discussions.  More precisely,  we first recall the  definitions of fractional integrals and derivatives.
We then collect  some important properties of  Jacobi polynomials and the related Jacobi-Gauss-Lobatto  interpolation.  
We also highlight in this section  the  transform between Jacobi polynomials with different parameters, which is related to   the so-called {\em connection problem}. 

\subsection{Fractional integrals and derivatives}  
Let $\mathbb N$ and $\mathbb R$ be the sets of positive integers and real numbers, respectively, and denote by 
 \begin{equation}\label{mathNR}
{\mathbb N}_0:=\{0\}\cup {\mathbb N}, \quad {\mathbb R}^+:= \big\{a\in {\mathbb R}: a> 0\big\}, \quad {\mathbb R}^+_0:=\{0\}\cup {\mathbb R}^+.
\end{equation}
The  definitions  of  fractional integrals and  fractional derivatives in the Caputo and Riemann-Liouville sense can be found from  many resources
(see,  e.g., \cite{Pod99,Diet10}): {\em  For  $\rho\in {\mathbb R}^+,$  the left-sided and right-sided   fractional integrals of order $\rho$ are defined  by 
 \begin{equation}\label{leftintRL}
 \begin{split}
   &({}_a I_{x}^\rho u)(x)=\frac 1 {\Gamma(\rho)}\int_{a}^x \frac{u(y)}{(x-y)^{1-\rho}} dy, \quad 
   ({}_x I_{b}^\rho u)(x)=\frac 1 {\Gamma(\rho)}\int_{x}^b \frac{u(y)}{(y-x)^{1-\rho}} dy,
   \end{split}
\end{equation}
for $x\in (a,b),$ respectively,  where $\Gamma(\cdot)$ is the Gamma function.}   
 
Denote the ordinary derivative by  $D^k=d^k/dx^k$ (with  $k\in {\mathbb N}$).    In general,  the fractional integral and  ordinary derivative operators are not  commutable, leading to  two types of fractional derivatives:     {\em For $\mu\in (k-1,k)$ with $k\in {\mathbb N},$ the 
left-sided Caputo fractional derivative  of order $\mu$ is  defined by
\begin{equation}\label{left-fra-der-c}
    \big({^C}\hspace*{-7pt}{}_aD_{\!x}^\mu\, u\big)(x)={}_a I_{x}^{k-\mu}\, \big(D^k u\big)(x)
    =\dfrac{1}{\Gamma(k-\mu)}\displaystyle\int_{a}^x\dfrac{u^{(k)}(y)}
    {(x-y)^{\mu-k+1}}dy,
\end{equation}
and the left-sided  Riemann-Liouville fractional derivative  of order $\mu$   defined by
\begin{equation}\label{left-fra-der-rl}
    \big({^{R}}\hspace*{-7pt}{}_aD_{\!x}^\mu\, u\big)(x)=D^k\big({}_a I_{x}^{k-\mu}\, u\big)(x)
   =\dfrac{1}{\Gamma(k-\mu)}\dfrac{d^k}{dx^k} \displaystyle\int_{a}^x\dfrac{u(y)}{(x-y)^{
    \mu-k+1}}dy.
\end{equation}}
Note that  if $\mu=k\in {\mathbb N},$  we have  ${^C}\hspace*{-7pt}{}_aD_{\!x}^k={^R}\hspace*{-7pt}{}_aD_{\!x}^k=D^k.$
\begin{rem}\label{rmkOpt} Similarly, one can define the right-sided Caputo and Riemann-Liouville derivatives: 
\begin{equation}\label{rightones}
\big({^C}\hspace*{-7pt}{}_xD_{\!b}^\mu u\big)(x)=(-1)^k{}_x I_{b}^{k-\mu}\,
\big(D^k u\big)(x),\quad \big({^R}\hspace*{-7pt}{}_xD_{\!b}^\mu u\big)(x)
= (-1)^k\,D^k\big({}_x I_{b}^{k-\mu}u\big)(x).
 \end{equation}
 With a change of variables:  $$x=b+a-t,\;\;\; u(x)=v(a+b-x),\quad x,t\in (a,b),$$   one finds 
 \begin{equation}\label{rightintRL1}
 \begin{split}
   &({}_t I_{b}^\rho v)(t)= ({}_a I_{x}^\rho u)(x),\quad x,t\in (a,b),
   \end{split}
\end{equation}
and likewise for the fractional derivatives.  In view of this, we restrict our discussions to the left-sided fractional integrals and derivatives. \qed
\end{rem}

Recall that  for $\mu\in (k-1,k)$ with $k\in {\mathbb N},$
\begin{equation}\label{DSc}
 \big({^{R}}\hspace*{-7pt}{}_aD_{\!x}^\mu\, u\big)(x) = \big({^C}\hspace*{-7pt}{}_aD_{\!x}^\mu\, u\big)(x)+ \sum_{j=0}^{k-1} \frac{u^{(j)}(a)}{\Gamma(1+j-\mu)}(x-a)^{j-\mu} 
\end{equation}
(see, e.g., \cite{Pod99,Diet10}),  which implies
\begin{equation}\label{deveqcase}
\big({^{R}}\hspace*{-7pt}{}_aD_{\!x}^\mu\, u\big)(x) =\big({^C}\hspace*{-7pt}{}_aD_{\!x}^\mu\, u\big)(x),\;\;\;   {\rm if}\;\; u^{(j)}(a)=0,\;\;\; j=0,\cdots,k-1.
\end{equation}
Moreover, there holds (see, e.g., \cite[Thm. 2.14]{Diet10}):
\begin{equation}\label{rulesa}
   {^{R}}\hspace*{-7pt}{}_a D_x^\mu\,{}_a I_{x}^{\mu}\, u(x) = u(x)\quad \text{a.e. in} \;\; (a,b),\;\;  \mu\in {\mathbb R}^+.
 \end{equation}
In addition, we have  the  explicit formulas  (see, e.g.,  \cite[P. 49]{Diet10}):   for  real $\eta>-1$ and $\mu\in{\mathbb R}^+,$
\begin{equation}\label{intformu}
{}_a I_{x}^\mu \, (x-a)^\eta=  \dfrac{\Gamma(\eta+1)}{\Gamma(\eta+\mu+1)} (x-a)^{\eta+\mu},
\end{equation}
and for $\mu \in (k-1,k)$ with $k\in {\mathbb N},$
\begin{equation}\label{propCaputo1}
{^C}\hspace*{-7pt}{}_aD_{\!x}^\mu\, (x-a)^\eta=
\begin{cases}0,\quad & {\rm if}\;\;  \eta \in \{0,1,\cdots, k-1\},\\[6pt]
 \dfrac{\Gamma(\eta+1)}{\Gamma(\eta-\mu+1)} (x-a)^{\eta-\mu},\;\;   & {\rm if}\;\;\;\eta>k-1, \;\; \eta\in {\mathbb R}.
\end{cases}
\end{equation}
Similarly,    for  $\mu \in (k-1,k)$ with $k\in {\mathbb N},$ and real $\eta>-1,$ we have  (cf. \cite[P. 72]{Pod99})
\begin{equation}\label{rlpower}
{^{R}}\hspace*{-7pt}{}_aD_{\!x}^\mu\, (x-a)^\eta=\frac{\Gamma(\eta+1)}{\Gamma(\eta-\mu+1)} (x-a)^{\eta-\mu}.
\end{equation}

Hereafter,  we restrict our attention to the interval $\Lambda:=(-1,1),$ and simply denote
\begin{equation}\label{simplifynota}
I_{-}^\mu:={}_{-1} I_{x}^\mu,\quad  {^C}\hspace*{-3pt}D_-^\mu:={}_{-1}\hspace*{-5pt}{^C}\hspace*{-2pt}D_{\!x}^\mu,\quad  {^R}\hspace*{-2pt} D_-^\mu={}_{-1}\hspace*{-6pt}{^{R}}\hspace*{-2pt}D_{\!x}^\mu\,,\quad x\in \Lambda.
\end{equation}
Apparently,  the formulas and results can be   extended to the general interval $(a,b)$ straightforwardly.

\subsection{Jacobi polynomials and Jacobi-Gauss-Lobatto interpolation}\label{subsect:Jcbi} 
 Throughout this paper,  the notation and normalization of Jacobi polynomials are  in accordance with  Szeg\"o  \cite{szeg75}.
 
 For $\alpha,\beta\in {\mathbb R},$ the Jacobi polynomials  are defined  by the hypergeometric function (cf. Szeg\"o \cite[(4.21.2)]{szeg75}):
\begin{equation}\label{Jacobidefn0}
\begin{split}
P_n^{(\alpha,\beta)}(x)&=\frac{\Gamma(n+\alpha+1)}{n!\Gamma(\alpha+1)}{}_2F_1\Big(-n, n+\alpha+\beta+1;\alpha+1;\frac{1-x} 2\Big),\;\;\; x\in \Lambda, \;\;  n\in {\mathbb N},\\
\end{split}
\end{equation}
and  $P_0^{(\alpha,\beta)}(x)\equiv 1.$ Note that
{\em $P_n^{(\alpha,\beta)}(x)$  is always a polynomial in $x$ for all  $\alpha,\beta\in {\mathbb R},$ but not always of degree $n.$}   A reduction of the degree of $P_n^{(\alpha,\beta)}(x)$ occurs   if and only if 
\begin{equation}\label{reduction}
m:=-(n+\alpha+\beta)\in {\mathbb N} \;\;\; {\rm and}\;\;\; 1\le m\le n
 \end{equation}
(cf. \cite[P. 64]{szeg75} and \cite{Cagliero2014}).
Note that for $\alpha,\beta\in {\mathbb R},$ there hold
 \begin{equation}\label{parity}
 P_n^{(\alpha,\beta)}(x)=(-1)^n P_n^{(\beta,\alpha)}(-x); \quad P_n^{(\alpha,\beta)}(1)=\frac{\Gamma(n+\alpha+1)}{n!\Gamma(\alpha+1)}.
 \end{equation}

For  $\alpha,\beta>-1,$   the classical  Jacobi polynomials  are orthogonal with respect to the Jacobi weight function:  $\omega^{(\alpha,\beta)}(x) = (1-x)^{\alpha}(1+x)^{\beta},$ namely,
\begin{equation}\label{jcbiorth}
    \int_{-1}^1 {P}_n^{(\alpha,\beta)}(x) {P}_{n'}^{(\alpha,\beta)}(x) \omega^{(\alpha,\beta)}(x) \, dx= \gamma _n^{(\alpha,\beta)} \delta_{nn'},
\end{equation}
where $\delta_{nn'}$ is the Dirac Delta symbol, and
\begin{equation}\label{co-gamma}
\gamma _n^{(\alpha,\beta)} =\frac{2^{\alpha+\beta+1}\Gamma(n+\alpha+1)\Gamma(n+\beta+1)}{(2n+\alpha+\beta+1) n!\,\Gamma(n+\alpha+\beta+1)}.
\end{equation}
However, the orthogonality 
does not carry over to the  general case with $\alpha$ or $\beta\le -1$ (see, e.g., 
 \cite{kuijlaars2005orthogonality} and \cite[Ch. 3]{koekoek2010hypergeometric}).


The following formulas derived from Bateman fractional integral formulas of Jacobi polynomials \cite{Bateman1909} (also see \cite[P.  313]{Andrews99},  \cite[P. 96]{szeg75} and \cite{CSW2014})  are dispensable for the algorithm development.
\begin{theorem}\label{JacobiForm2} Let  $\rho, s\in  {\mathbb R}^+, \; n\in {\mathbb N}_0$ and $x\in \Lambda.$ Then for     $ \alpha\in {\mathbb R}$ and
$\beta>-1,$ we have
 \begin{equation}\label{newbatemanam}
I_{-}^\rho\big\{(1+x)^\beta P_n^{(\alpha,\beta)}(x)\big\}=\frac{\Gamma(n+\beta+1)}{\Gamma(n+\beta+\rho+1)}
(1+x)^{\beta+\rho} P_n^{(\alpha-\rho,\beta+\rho)}(x),
\end{equation}
and
\begin{equation}\label{newbatemanam3s}
{}^R\hspace*{-2pt}D_-^s\big\{(1+x)^{\beta+s} P_n^{(\alpha-s,\beta+s)}(x)\big\}
=\frac{\Gamma(n+\beta+s+1)} {\Gamma(n+\beta+1)}(1+x)^\beta P_n^{(\alpha,\beta)}(x).
\end{equation}
\end{theorem}

As direct consequences of  Theorem \ref{JacobiForm2},  we have the following  important special cases. 
\begin{corollary}\label{twospecial} For $ \alpha\in {\mathbb R}, \rho\in {\mathbb R}^+,
 n\in {\mathbb N}_0$ and $x\in \Lambda,$
 \begin{align}
& I_{-}^\rho\big\{P_n^{(\alpha,0)}(x)\big\}=\frac{n!}{\Gamma(n+\rho+1)}
(1+x)^{\rho} P_n^{(\alpha-\rho,\rho)}(x); \label{specaseA}\\
&{}^R\hspace*{-2pt}D_-^\rho\big\{(1+x)^\rho P_n^{(\alpha,\rho)}(x)\big\}
=\frac{\Gamma(n+\rho+1)} {n!}P_n^{(\alpha+\rho,0)}(x). \label{specaseB}
\end{align}
In particular,  for $\rho\in {\mathbb R}^+,  n\in {\mathbb N}_0$ and $x\in \Lambda,$
 \begin{align}
& I_{-}^\rho\big\{P_n(x)\big\}=\frac{n!}{\Gamma(n+\rho+1)}
(1+x)^{\rho} P_n^{(-\rho,\rho)}(x); \label{specaseAB}\\
&{}^R\hspace*{-2pt}D_-^\rho\big\{(1+x)^\rho P_n^{(-\rho,\rho)}(x)\big\}
=\frac{\Gamma(n+\rho+1)} {n!}P_n(x). \label{specaseAB2}
\end{align}
\end{corollary}
\begin{rem}\label{imprmk} Remarkably, the formulas \eqref{specaseAB}-\eqref{specaseAB2} link up the Legendre polynomials with the non-polynomials  $(1+x)^\rho P_n^{(-\rho,\rho)}(x).$  They
 are referred to as the generalised Jacobi functions \cite{Guo.SW09,CSW2014}, and as the Jacobi poly-fractonomials \cite{zayernouri2013fractional} when $0<\rho<1$. \qed 
 \end{rem}
 

For $\alpha,\beta>-1,$ let $\big\{x_j:=x_{N,j}^{(\alpha,\beta)},\omega_j:=\omega_{N,j}^{(\alpha,\beta)}\big\}_{j=0}^N$   be the set of  Jacobi-Gauss-Lobatto  (JGL) quadrature nodes and weights,  where the nodes are zeros of $(1-x^2) D P_N^{(\alpha,\beta)}(x).$   Hereafter,  we assume that $\{x_j\}$ are  arranged in ascending order so that  $x_0=-1$ and $x_N=1.$  Moreover,  to alleviate the burden of heavy notation, we sometimes drop the parameters $\alpha, \beta $ in the notation, whenever it is clear from the context.


The JGL quadrature  enjoys the exactness  (see, e.g.,  \cite[Ch. 3]{ShenTangWang2011}):  
\begin{equation}\label{newquad}
\int_{-1}^1 \phi(x) \omega^{(\alpha,\beta)}(x)\,dx=\sum_{j=0}^N \phi(x_j)\omega_j,\quad \forall\, \phi\in {\mathcal  P}_{2N-1},
\end{equation}
where ${\mathcal P_N}$ is the set of all  polynomials of degree at most $N.$  Let  ${\mathcal I}_N u$ be the JGL Lagrange polynomial interplant of $u\in C(\bar\Lambda)$ defined by 
\begin{equation}\label{Inuexp}
\big({\mathcal I}_Nu\big)(x)=\sum_{j=0}^N u(x_j) h_j(x)\in {\mathcal P}_N, 
\end{equation} 
where the interpolating basis polynomials $\{h_j\}_{j=0}^N$  can be expressed by 
\begin{equation}\label{jacbasisfun}
h_j(x)=\sum_{n=0}^N  t_{nj} P_n^{(\af,\bt)}(x),\quad  0\le j\le N,\;\;\; {\rm where}\;\;\;  t_{nj}:= \frac{\omega_j}{\tilde \gamma_n^{(\af,\bt)}} P_n^{(\af,\bt)}(x_j),
\end{equation}
with 
\begin{equation}\label{tildgamma}
\tilde \gamma_n^{(\af,\bt)}=\gamma_n^{(\af,\bt)}, \;\; 0\le n\le N-1; \;\;\; \tilde \gamma_N^{(\af,\bt)}=\Big( 2+\frac {\af+\bt+1} N \Big)\gamma_N^{(\af,\bt)}.
\end{equation}

\subsection{Transform between Jacobi polynomials with different parameters}\label{connprob} 
 Our efficient computation of fractional differentiation matrices and their inverses,  relies on 
the transform between Jacobi expansions with different parameters.  It is evident that for  $ \af,\bt, a,b>-1,$ 
$$
{\mathcal P}_N={\rm span}\big\{P_n^{(\af,\bt)}\,:\, 0\le n\le N\big\}={\rm span}\big\{P_l^{(a,b)}\,:\, 0\le l\le N\big\}. 
$$
  {\em  
Given the Jacobi expansion coefficients $\{\hat u_n^{(\af,\bt)}\}$ of $u\in {\mathcal P}_N$, find the coefficients $\{\hat u_l^{(a,b)}\}$ such that}
\begin{equation}\label{uluk}
u(x)=\sum_{n=0}^N \hat u_n^{(\af,\bt)} P_n^{(\af,\bt)}(x)=\sum_{l=0}^N \hat u_l^{(a,b)} P_l^{(a,b)}(x).  
\end{equation}
This  defines a {\em connection problem} (cf. \cite{Askey75}) resolved by the transform:
\begin{equation}\label{conntrans}
\bs {\hat u}^{(a,b)}= {}^{(\af,\bt)\!}\bs C^{(a,b)}\, \bs {\hat u}^{(\af,\bt)},
\end{equation}
where $\bs {\hat u}^{(\af,\bt)}$ and $\bs {\hat u}^{(a,b)}$ are column-$(N+1)$ vectors of the coefficients, and ${}^{(\af,\bt)\!}\bs C^{(a,b)}$ is the connection matrix  of the transform from $\big\{P_n^{(\af,\bt)}\big\}$ to 
 $\big\{P_l^{(a,b)}\big\}.$ One finds from  the orthogonality  \eqref{jcbiorth} and \eqref{uluk}  that the entries of  ${}^{(\af,\bt)\!}\bs C^{(a,b)}$, i.e.,  the connection coefficients, are given by 
\begin{equation}\label{constAB}
{}^{(\af,\bt)\!}\bs C_{ln}^{(a,b)}:=\frac 1 {\gamma_l^{(a,b)}} 
\int_{-1}^1P_l^{(a,b)}(x) P_n^{(\af,\bt)}(x)\, \omega^{(a,b)}(x) dx.
\end{equation}
Some remarks are in order.
\begin{itemize}

\item By the orthogonality  \eqref{jcbiorth},  we have  ${}^{(\af,\bt)\!}\bs C_{ln}^{(a,b)}=0$ 
for  $ n<l,$ so the connection matrix 
is an upper triangular matrix. Therefore, \eqref{conntrans} yields 
\begin{equation}\label{ucoefrela}
\hat u_l^{(a,b)}
=\sum_{n=l}^N {}^{(\af,\bt)\!}\bs C_{ln}^{(a,b)}\,   \hat u_n^{(\af,\bt)},\;\;\;\; 0\le l\le N.
\end{equation}

\item In fact, we have the explicit  formula of the connection coefficient (cf. \cite[P. 357]{Andrews99}) 
\begin{equation}\label{cocoefcnk}
\begin{split}
{}^{(\af,\bt)\!}\bs C_{ln}^{(a,b)}&=
(2l+a+b+1) \frac{\Gamma(n+\af+1)}{\Gamma(n+\af+\bt+1)} \frac{\Gamma(l+a+b+1)}{\Gamma(l+a+1)}\times\\
&\quad  \sum_{m=0}^{n-l}\frac{(-1)^m \Gamma(n+l+m+\af+\bt+1) \Gamma(m+l+a+1)}{m!(n-l-m)!
\Gamma(l+m+\af+1)\Gamma(m+2l+a+b+2)}.
\end{split}
\end{equation}
 This  exact formula is less useful in computation,   as  even in the Chebyshev-to-Legendre case, significant effort has to be made to analyze their behaviors  and  take care of the cancellations, when $N$ is large  (cf. \cite{Alpert.R91,Boyd.14}).    One can actually compute the connection coefficients by using the Jacobi-Gauss quadrature  with $(N+1)$ nodes  and with respect to the weight function $\omega^{(a,b)}.$

\item In general,  it requires $O(N^2)$ operations to carry out the matrix-vector product in \eqref{conntrans}.  In practice, several techniques have been proposed to speed up the transforms 
(see, e.g., \cite{Alpert.R91,AkHes12,Bel.R14,HSX2015} and the monograph \cite{keiner2011fast} and the references therein). In particular,  through exploiting the remarkable property that  the columns of  the connection matrix are eigenvectors of a certain structured quasi-separable  matrix,  fast and stable algorithms  can be developed  (cf. \cite{keiner2011fast,Bel.R14} and the references therein). The interesting work \cite{HSX2015} fully used the low-rank property of the connection matrix, and proposed fast algorithms based on  rank structured matrix approximation. 


\end{itemize}

%
%
%
%
%


\vskip 10pt

\section{Fractional  pseudospectral differentiation}\label{sect:FPSDM}
\setcounter{equation}{0}
\setcounter{lmm}{0}
\setcounter{thm}{0}

In this section, we extend the pseudospectral differentiation (PSD) process of integer order derivatives to the fractional context, and present efficient algorithms  for computing the fractional pseudospectral differentiation matrix (F-PSDM). We show that
\begin{itemize}
\item[(i)] in the Caputo case,  it suffices to evaluate Caputo  F-PSDM of order $\mu\in (0,1)$ to compute F-PSDM of any order (see Theorem \ref{Thm:recurvere}); 
\item[(ii)] in the Riemann-Liouville case, it is necessary to modify the fractional derivative operator in order to absorb the singular fractional factor (see \eqref{fpsdmod}), and the computation of the modified F-PSDM of any order boils down to computing a modified  fractional  integral matrix of order $\mu\in (0,1)$   (see Theorem \ref{Thm:recurvereRL}).  
\end{itemize}

\subsection{Fractional pseudospectral differentiation process}
It is known that the pseudospectral differentiation process   is the heart of a collocation/pseudospectral method for PDEs (see, e.g., \cite{CHQZ06,ShenTangWang2011}). 
Typically, for any $u\in {\mathcal P}_N,$ the differentiation  $D^k u$ is carried out  via  \eqref{Inuexp} in an exact manner, that is, 
\begin{equation}\label{DIn}
D^k u(x)=\sum_{j=0}^N u(x_j) D^kh_j(x), \quad k\in {\mathbb N}.
\end{equation}
It is straightforward to extend this to the fractional pseudospectral differentiation.  More precisely,   for any $u\in {\mathcal P_N},$  
\begin{equation}\label{fpsd}
(D^\mu u)(x)=\sum_{j=0}^N u(x_j) D^\mu h_j(x), \quad D^\mu:={^C}\hspace*{-3pt}D_-^\mu, {^R}\hspace*{-3pt}D_-^\mu,\;\;\; \mu\in {\mathbb R}^+.
\end{equation}
However, in distinct contrast to \eqref{DIn},  we have  $D^\mu u,D^\mu h_j \not \in \mathcal P_N,$ if  $\mu\not \in {\mathbb N}.$  To  provide some  insights into this, we introduce 
the space: 
\begin{equation}\label{fspsp}
{\mathcal F}_N^{(\nu)}:=\big\{ (1+x)^\nu \phi\;:\;\forall\, \phi\in{\mathcal  P}_N \big\},\quad \nu\in {\mathbb R},
\end{equation}
and show the following properties. 
\begin{lmm}\label{newsolut} For $\mu\in (k-1,k)$ with $k\in {\mathbb N},$ and for  any $u\in {\mathcal P}_N,$ we have  
\begin{equation}\label{Dcrprop}
{^C}\hspace*{-3pt}D_-^\mu u\in {\mathcal F}_{N-k}^{(k-\mu)},\quad  {^R}\hspace*{-2pt}D_-^\mu u\in {\mathcal F}_{N}^{(-\mu)},
\end{equation}
and 
\begin{equation}\label{singular}
{^C}\hspace*{-3pt}D_-^\mu u\to 0,\quad {^R}\hspace*{-3pt}D_-^\mu u\to \infty\quad {\rm as}\;\;  x\to -1.
\end{equation}
\end{lmm}
\begin{proof} It is clear that 
$$D^k u\in {\mathcal P}_{N-k}={\rm span}\big\{P_n\;:\; 0\le n\le N-k\big\}.$$
Thus,  we derive from the definition \eqref{left-fra-der-c} and \eqref{specaseAB} with $\rho=k-\mu$ that  
\begin{equation}\label{casesA}
{^C}\hspace*{-3pt}D_-^\mu u=I_-^{k-\mu} (D^ku) =(1+x)^{k-\mu}\phi,\quad \text{for some}\;\; \phi\in {\mathcal P}_{N-k}.
\end{equation}
Similarly, in the Riemann-Liouville  case,   we deduce from \eqref{specaseAB} that   $I_-^{k-\mu}u\in {\mathcal F_N^{(k-\mu)}}.$ Then  by the definition  \eqref{left-fra-der-rl}, we obtain from  a direct calculation that 
\begin{equation}\label{casesB}
{^R}\hspace*{-3pt}D_-^\mu u=D^k(I_-^{k-\mu} u)=(1+x)^{-\mu} \psi, \quad \text{for some}\;\; \psi\in {\mathcal P}_{N}.
\end{equation}
Thus,  \eqref{Dcrprop} is verified, from which \eqref{singular} follows immediately. 
\end{proof}

\begin{rem}\label{singularrmk} This implication of  Lemma \ref{newsolut} is that 
\begin{itemize}
\item[(i)]  if a FDE has a smooth solution, the  source term might have a singular behaviour; 
 \item[(ii)] 
conversely, for a  FDE with smooth inputs,  the solution might possess  singularity. 
\end{itemize}
To achieve  spectrally accurate approximation for some prototype FDEs pertaining to  the latter case,  the recent works \cite{zayernouri2013fractional,CSW2014}  proposed to approximate the singular solutions by  using Jacobi polyfractonomials and general Jacobi functions, i.e., the basis of ${\mathcal F}_N^{(\nu)}.$  \qed
\end{rem}

Observe from \eqref{singular} that  the Riemann-Liouville fractional derivative  of any polynomial tends to infinity as  $x_0\to -1.$  
This brings about  some inconvenience for the computation of the related F-PSDM and  implementation of the collocation scheme.  
This inspires us to multiply both sides of \eqref{fpsd}  by the singular factor $(1+x)^\mu,$ leading to the modified   Riemann-Liouville fractional  pseudospectral differentiation:
\begin{equation}\label{fpsdmod}
\big({^R}\hspace*{-2pt} {\widehat D}^{\mu}_-u\big)(x)=\sum_{j=0}^N u(x_j) \big({^R}\hspace*{-2pt} {\widehat D}^{\mu}_- h_j\big)(x)\;\;\;\; {\rm where}\;\;\;  
{^R}\hspace*{-2pt} {\widehat D}^{\mu}_-:= (1+x)^\mu {^R}\hspace*{-2pt}{D}^{\mu}.
\end{equation}
With such a modification, we can recover the Riemann-Liouville fractional derivative values at $x_i\not=-1$ by
\begin{equation}\label{addresul2}
\big({^R}\hspace*{-2pt} {D}^{\mu}_- u\big)(x_i)=(1+x_i)^{-\mu}\big({^R}\hspace*{-2pt} {\widehat D}^{\mu}_-u\big)(x_i),\quad 1\le i\le N. 
 \end{equation}
 Correspondingly,  we can  define the modified factional integral and state some important properties as follows. 
 \begin{lmm}\label{Imodify}   Let $u\in{\mathcal P}_N$ and $\{h_j\}$ be the Lagrange interpolating basis polynomials at JGL points as before, and define 
 \begin{equation}\label{eqnasd}
 \hat I_-^\mu =(1+x)^{-\mu} I_-^\mu,\quad {^R}\hspace*{-2pt} {\widehat D}^{\mu}_-:= (1+x)^\mu\, {^R}\hspace*{-2pt}{D}^{\mu}_-,\quad\forall\, \mu\in {\mathbb R}^+.
 \end{equation} 
Then we have 
 \begin{equation}\label{newpropC}
 \hat I_-^\mu u, {^R}\hspace*{-2pt} {\widehat D}^{\mu}_-u\in {\mathcal P}_N,\quad {}_0{\mathcal P}_N={\rm span}\big\{\hat I_-^\mu h_j\,:\, 1\le j\le N \big\},
 \end{equation}
 where ${}_0{\mathcal P}_N=\{\phi\in {\mathcal P}_N\,:\, \phi(-1)=0\}.$
 \end{lmm}
 \begin{proof} It is clear that by \eqref{casesB},  ${^R}\hspace*{-2pt} {\widehat D}^{\mu}_-u\in {\mathcal P_N},$ and  by \eqref{specaseAB} and \eqref{eqnasd},
 \begin{equation}\label{proppf}
 \hat I_{-}^\mu\big\{P_n(x)\big\}=\frac{n!}{\Gamma(n+\mu+1)}
 P_n^{(-\mu,\mu)}(x),\quad \mu\in {\mathbb R}^+.
 \end{equation}
 Note that for any real $\mu>0,$ $P_n^{(-\mu,\mu)}(x)$ is a polynomial of degree $n$  (cf. \cite[P. 64]{szeg75}).  Thus, we have 
 \begin{equation}\label{IPn}
 {\mathcal P}_N={\rm span}\big\{\hat I_{-}^\mu P_n\,:\,0\le n\le N\big\}= {\rm span}\big\{\hat I_{-}^\mu h_j\,:\,0\le j\le N\big\},
 \end{equation}
 and $\hat I_-^\mu u\in {\mathcal P}_N.$
 
We now  show   $(\hat I_-^\mu h_j)(-1)=0$ for $1\le j\le N.$ It is clear that $\big\{(1+x)P_n^{(\mu,1)}\big\}_{n=0}^{N-1}$ forms a basis of ${}_0{\mathcal P}_N,$  and by  \eqref{newbatemanam} with $\alpha=\mu$ and $\beta=1,$ 
\begin{equation}\label{newbatemanamC}
\hat I_{-}^\mu \big\{(1+x) P_n^{(\mu,1)}(x)\big\}=\frac{(n+1)!}{\Gamma(n+\mu+2)}
(1+x)P_n^{(0,1+\mu)}(x).
\end{equation}
Since $h_j\in {}_0{\mathcal P}_N,$  the identity \eqref{newbatemanamC} implies   $(\hat I_-^\mu h_j)(-1)=0$ for $1\le j\le N.$
 \end{proof}

\subsection{Caputo fractional pseudospectral differentiation matrices}\label{sub:matrixDiff}
As before, we use boldface uppercase  (resp. lowercase) letters   to denote matrices (resp. vectors), and simply  denote the entries  of a matrix $\bs A$ by $\bs A_{ij}.$   
Introduce the Caputo  F-PSDM of order $\mu:$  
\begin{equation}\label{cFPSDM}
{^C}\hspace*{-3pt}{\bs D}^{(\mu)}\in{\mathbb R}^{(N+1)\times (N+1)},\quad {^C}\hspace*{-3pt}{\bs D}^{(\mu)}_{ij}={^C}\hspace*{-3pt}D_-^\mu h_j(x_i).
\end{equation}
In particular,  for $\mu=k\in {\mathbb N},$ we denote $\bs D^{(k)}={^C}\hspace*{-3pt}{\bs D}^{(k)}$ and $\bs D={\bs D}^{(1)}.$ 

Remarkably,   the higher order Caputo fractional PSDM at JGL points  can be computed by using the following  recursive relation.
\begin{thm}\label{Thm:recurvere} Let  $\mu\in(0,1).$  Then we have
\begin{equation}\label{Drecu}
{}^C\hspace*{-3pt}{\bs D}^{(k+\mu)}= {}^C\hspace*{-3pt}{\bs D}^{(\mu)} {\bs D}^{(k)}={}^C\hspace*{-3pt}\bs D^{(\mu)}{\bs D}^k,\;\;\; \;  k\in{\mathbb N},
\end{equation}
where $\bs D^k$ stands for  the product of $k$ copies of the first-order PSDM at JGL points.  
\end{thm}
\begin{proof}  For any $u\in {\mathcal P}_N,$ we have
\begin{equation}\label{upN}
u'(x)=\sum_{l=0}^N u(x_l) h_l'(x).
\end{equation}
Taking $u=h'_j$ in \eqref{upN},  leads to
\begin{equation}\label{upN2}
h_j''(x)=\sum_{l=0}^N h_j'(x_l) h_l'(x),
\end{equation}
which, together with  the definition \eqref{left-fra-der-c},  implies
\begin{equation}\label{upN3}
{}^C\hspace*{-3pt}{ D}^{1+\mu}_- h_j(x)=I_-^{1-\mu} h_j''(x) =\sum_{l=0}^N h_j'(x_l) I_-^{1-\mu}h_l'(x)=
\sum_{l=0}^N h_j'(x_l) {}^C\hspace*{-3pt}{ D}^{\mu}_-h_l(x).
\end{equation}
Taking $x=x_i$ in the above, we obtain the matrix identity:
\begin{equation}\label{Drecu2}
{}^C\hspace*{-3pt}{\bs D}^{(1+\mu)}= {}^C\hspace*{-3pt}{\bs D}^{(\mu)}{\bs D},\quad \mu\in (0,1). 
\end{equation}
This leads to \eqref{Drecu} with $k=1.$
Taking $u=h_j^{(k)}(x)$ in \eqref{upN},  we can derive the first identity in \eqref{Drecu} in the same fashion.

Using the property (see \cite[Thm. 3.10]{ShenTangWang2011}):
\begin{equation}\label{ArelationPSDM}
{\bs D}^{(k)}=\bs D^k,\quad k\in {\mathbb N},
\end{equation}
 we obtain  the second identity in \eqref{Drecu}. 
\end{proof}

It is seen from Theorem \ref{Thm:recurvere}   that the computation of Caputo F-PSDM of any order at JGL points boils down to computing  the first-order usual PSDM ${\bs D}$ (whose explicit formula can be found in e.g., \cite{ShenTangWang2011}), and the Caputo F-PSDM ${}^C\hspace*{-3pt}{\bs D}^{(\mu)}$  with $\mu\in (0,1).$  We present the formulas below. 
\begin{thm}\label{JacobimatCa}  Let $\big\{x_j=x_{N,j}^{(\alpha,\beta)}\big\}_{j=0}^N$ with $\alpha,\beta>-1$ and $x_0=-1$ be the JGL points, and let 
$\big\{\omega_j=\omega_{N,j}^{(\alpha,\beta)}\big\}_{j=0}^N$ be the corresponding quadrature weights.   Then the entries of  
${}^C\hspace*{-3pt}{\bs D}^{(\mu)}$  with $\mu\in (0,1)$  can be computed   by 
\begin{align}\label{DmuJGL}
 {}^C\hspace*{-3pt}{\bs D}^{(\mu)}_{ij}
=(1+x_{i})^{1-\mu} \sum_{l=1}^N    \frac{(l-1)!}{\Gamma(l+1-\mu)}\, s_{lj}\,  P_{l-1}^{(\mu-1,1-\mu)}(x_{i})\,, 
\end{align}
for $0\le i,j\le N,$ where
\begin{align}\label{slj2}
 s_{lj}=\frac 1 2  \sum_{n=l-1}^N  (n+\alpha+\beta+1)  {}^{(\af+1,\bt+1)\!}\bs C_{l-1,n-1}^{(0,0)}\, t_{nj},\;\;\;   t_{nj}:= \frac{\omega_j}{\tilde \gamma_n^{(\af,\bt)}} P_n^{(\af,\bt)}(x_j),
\end{align}
 $\big\{{}^{(\af+1,\bt+1)\!}\bs C_{l-1,n-1}^{(0,0)} \big\}$ are the Jacobi-to-Legendre connection coefficients,  and  $\{\tilde \gamma_n^{(\af,\bt)}\}$ are  defined in \eqref{tildgamma}.  In particular, for $\alpha=\beta=0,$  we can alternatively compute the coefficients $\{s_{lj}\}$ by 
\begin{equation}\label{tnjnj}
  s_{lj}=\frac 1{\gamma_{l-1}}\big\{ \delta_{jN}+(-1)^l \delta_{j0}- \omega_j P_{l-1}'(x_j)\big\}, \quad \gamma_{l-1}=\frac 2 {2l-1}.
\end{equation}
 \end{thm}
 To avoid the distraction from the main results, we provide the derivation of the formulas in Appendix 
 \ref{AppendixB}.

\begin{rem}\label{Newcase}  We see that in the Legendre case, we can bypass the connection problem.  It is  noteworthy that in \cite{LZL12},   the Caputo F-PSDM of order $\mu>0$ was computed largely by the derivative formula of $P_n$ and some recurrence relation of 
$I_-^\rho P_n$ built upon three-term recurrence formula of Legendre polynomials.
As shown above, the use of the compact, explicit formula \eqref{specaseAB}   leads to much concise representation and  stable computation. \qed  
   \end{rem}

\subsection{Modified  Riemann-Liouville  fractional pseudospectral differentiation matrices} 
 Introduce the matrices: 
\begin{equation}\label{newnotationA}
{^R}\hspace*{-2pt}\bs {\widehat D}^{(\mu)}, \bs {\hat  I}^{(\mu)}\in {\mathbb R}^{(N+1)\times(N+1)}\;\; {\rm where}\;\;  {^R}\hspace*{-2pt}\bs {\widehat D}^{(\mu)}_{ij}= \big({^R}\hspace*{-2pt}{\widehat D}^{(\mu)}_-h_j\big)(x_i),\;\;   \bs {\hat  I}^{(\mu)}_{ij}=\big(\hat I^\mu_{-} h_j\big)(x_i). 
\end{equation}
We can show the following important property  similar to Theorem  \ref{Thm:recurvere}.
\begin{thm}\label{Thm:recurvereRL} Let $\{h_j\}$ be the JGL interpolating basis polynomials.   Then for  $\mu\in (k-1,k)$ with $k\in {\mathbb N},$    we have 
\begin{equation}\label{DrecuRL}
{^R}\hspace*{-2pt}\bs {\widehat D}^{(\mu)}= \bs {\breve D}^{(k)} \bs {\hat I}^{(k-\mu)}, 
\end{equation}
where the entries of $\bs {\breve D}^{(k)}$ are given by 
\begin{equation}\label{entriDk}
\bs {\breve D}^{(k)}_{ij}=(1+x)^\mu D^k\big\{(1+x)^{k-\mu} h_j(x)\big\}\big|_{x=x_i},\quad 0\le i,j\le N. 
\end{equation} 
\end{thm}
\begin{proof} By \eqref{newpropC}, we can write that  for any $u\in {\mathcal P}_N,$  
\begin{equation}\label{newexpA}
(\hat I^{k-\mu}_-u)(x)=\sum_{l=0}^N (\hat I^{k-\mu}_-u)(x_l) h_l(x),
\end{equation}
Multiplying both sides by $(1+x)^{k-\mu},$ and using \eqref{eqnasd}, we find 
\begin{equation}\label{newexpB}
 (I^{k-\mu}_-u)(x)=\sum_{l=0}^N (\hat I^{k-\mu}_-u)(x_l) (1+x)^{k-\mu} h_l(x), 
\end{equation}
which implies 
\begin{equation}\label{newexpC}
 \big({^R}\hspace*{-2pt} {D}^{(\mu)}_- u\big)(x)= D^k (I^{k-\mu}_-u)(x)=\sum_{l=0}^N (\hat I^{k-\mu}_-u)(x_l) D^k\big\{(1+x)^{k-\mu} h_l(x)\big\},
\end{equation}
for  $x\in (-1,1].$ To remove the singularity at $x=-1,$ we multiply both sides of \eqref{newexpC}  by  $(1+x)^\mu,$ 
and reformulate the resulted identity by the modified operator  in \eqref{fpsdmod}, leading to  
\begin{equation}\label{newexpD}
\big({^R}\hspace*{-2pt} {\widehat D}^{(\mu)}_-u\big)(x)= \sum_{l=0}^N (\hat I^{k-\mu}_-u)(x_l) \Big\{(1+x)^\mu D^k\big\{(1+x)^{k-\mu} h_l(x)\big\}\Big\}.
\end{equation}
Taking $u=h_j$ and $x=x_i$ in the above equation yields  \eqref{DrecuRL}. 
\end{proof}

Observe from  \eqref{DrecuRL} that  it suffices to compute the modified fractional integral matrix $\bs {\hat I}^{(k-\mu)}$ with $k-\mu\in (0,1),$ since  $\bs {\breve D}^{(k)}$ can be expressed   in terms of the PSDM of integer order, e.g., for $k=1,$
\begin{equation}\label{k11mat}
\bs {\breve D}^{(1)}=(1-\mu)\bs I_{N+1}+\bs \Lambda \bs D, \quad \bs \Lambda={\rm diag}\big((1+x_0),\cdots, (1+x_N)\big), 
\end{equation}
where $\bs I_{N+1}$ is an identity matrix. 
%
\begin{thm}\label{JacobimatRL}  Let $\big\{x_j=x_{N,j}^{(\alpha,\beta)}\big\}_{j=0}^N$ with $\alpha,\beta>-1$ and $x_0=-1$ be the JGL points, and let 
$\big\{\omega_j=\omega_{N,j}^{(\alpha,\beta)}\big\}_{j=0}^N$ be the corresponding quadrature weights.   Then the entries of  $\bs {\hat I}^{(\mu)}$  with $\mu\in (0,1)$  can be computed   by 
\begin{align}\label{ImuJGL}
& \bs {\hat I}^{(\mu)}_{ij}=\sum_{l=0}^N  \frac{l!\, }{\Gamma(l+\mu+1)} \, \hat s_{lj}  P_l^{(-\mu,\mu)}(x_{i}),\quad 0\le i,j\le N,  
\end{align}
where
\begin{align}\label{slj2b}
\hat s_{lj}=\sum_{n=l}^N {}^{(\af,\bt)\!}C_{ln}^{(0,0)} t_{nj},\quad  
 t_{nj}:= \frac{\omega_j}{\tilde \gamma_n^{(\af,\bt)}} P_n^{(\af,\bt)}(x_j),
\end{align}
with $\big\{{}^{(\af,\bt)\!}C_{ln}^{(0,0)}\big\}$ being the Jacobi-to-Legendre connection coefficients,   and  $\tilde \gamma_n^{(\af,\bt)}$  defined in \eqref{tildgamma}.  In particular, if $\alpha=\beta=0,$ we have $\hat s_{lj}=t_{lj}.$
\end{thm}
\begin{proof} It is essential to use the explicit formulas  in Corollary  \ref{twospecial}.  Accordingly, we  expand the  JGL Lagrange interpolating basis  polynomials  $\{h_j\}$ in terms of 
Legendre polynomials, and  resort to the connection problem to transform between the bases   as before. 
Equating  \eqref{jacbasisfun} and the new expansion leads to  
\begin{equation}\label{matrxA}
h_j(x)=\sum_{n=0}^N  t_{nj} P_n^{(\af,\bt)}(x)=\sum_{l=0}^N \hat s_{lj} P_l(x),\quad 0\le j\le N,
\end{equation} 
which defines a connection problem. Thus by 
\eqref{ucoefrela},  
\begin{equation}\label{exastD}
\hat s_{lj}=
 \sum_{n=0}^N {}^{(\af,\bt)\!}C_{ln}^{(0,0)} t_{nj}=\sum_{n=l}^N {}^{(\af,\bt)\!}C_{ln}^{(0,0)} t_{nj},
\end{equation}
where we used the property: ${}^{(\af,\bt)\!}C_{ln}^{(0,0)} =0$ if $n<l.$ Then  it follows from \eqref{proppf} immediately that for $\mu\in (0,1),$
\begin{equation}\label{JGLImu}
\bs {\hat I}^{(\mu)}_{ij}=\big(\hat I_-^\mu h_j\big)(x_i)=\sum_{l=0}^N  \frac{l!\, \hat s_{lj}}{\Gamma(l+\mu+1)}  P_l^{(-\mu,\mu)}(x_{i}),\quad 0\le i,j\le N. 
\end{equation}
This leads to the desired formulas.  

In the Legendre case,   it is clear that the expansions in \eqref{matrxA} are identical, so we have $\hat s_{lj}=t_{lj}.$
\end{proof}

We conclude this section by providing some numerical study of (discrete) eigenvalues of  F-PSDMs.    Observe from \eqref{singular} that the first row of ${^C}\hspace*{-3pt}{\bs D}^{(\mu)}$ is entirely zero, so ${^C}\hspace*{-3pt}{\bs D}^{(\mu)}$ is always  singular.  We therefore remove the ``boundary" row/column, and define 
\begin{equation}\label{submatix}
{}^C\hspace*{-3pt}{\bs D}_{{\rm in}}^{(\mu)}:=\begin{cases}
\big({^C}\hspace*{-3pt}{\bs D}^{(\mu)}_{ij}\big)_{1\le i, j\le N}, \quad& {\rm if}\;\; \mu\in (0,1), \\[6pt]
\big({^C}\hspace*{-3pt}{\bs D}^{(\mu)}_{ij}\big)_{1\le i, j\le N-1}, \quad& {\rm if}\;\; \mu\in (1,2),
\end{cases}
\quad{\rm where}\;\; {^C}\hspace*{-3pt}\bs {D}^{(\mu)}_{ij}= \big({^C}\hspace*{-3pt}{D}^{\mu}_-h_j\big)(x_i), 
\end{equation}
which is invertible and  allows for   incorporating  boundary condition(s).  Similarly,  we define 
\begin{equation}\label{submatixRL}
{}^R\hspace*{-2pt}\widehat{\bs D}_{{\rm in}}^{(\mu)}:=\begin{cases}
\big({^R}\hspace*{-2pt}\widehat{\bs D}^{(\mu)}_{ij}\big)_{1\le i, j\le N}, \; & {\rm if}\; \mu\in (0,1), \\[6pt]
\big({^R}\hspace*{-2pt}\widehat{\bs D}^{(\mu)}_{ij}\big)_{1\le i, j\le N-1}, \; & {\rm if}\; \mu\in (1,2), 
\end{cases}
\quad \text{where\;  ${^R}\hspace*{-2pt}\widehat {\bs D}^{(\mu)}_{ij}= \big({^R}\hspace*{-2pt}{\widehat D}^{\mu}_-h_j\big)(x_i).$}
\end{equation}
In Figure \ref{uhatvsuNhat00}, we illustrate the smallest and largest eigenvalues (in modulus) of these matrices. Observe that in both cases, 
the largest eigenvalue grows like $O(N^{2\mu}),$ while the smallest one remains a constant in  the Caputo case, and mildly decays with respect to $N$ in the modified RL case.

\begin{figure}[h!]
  \centering
    \includegraphics[width=0.48\textwidth]{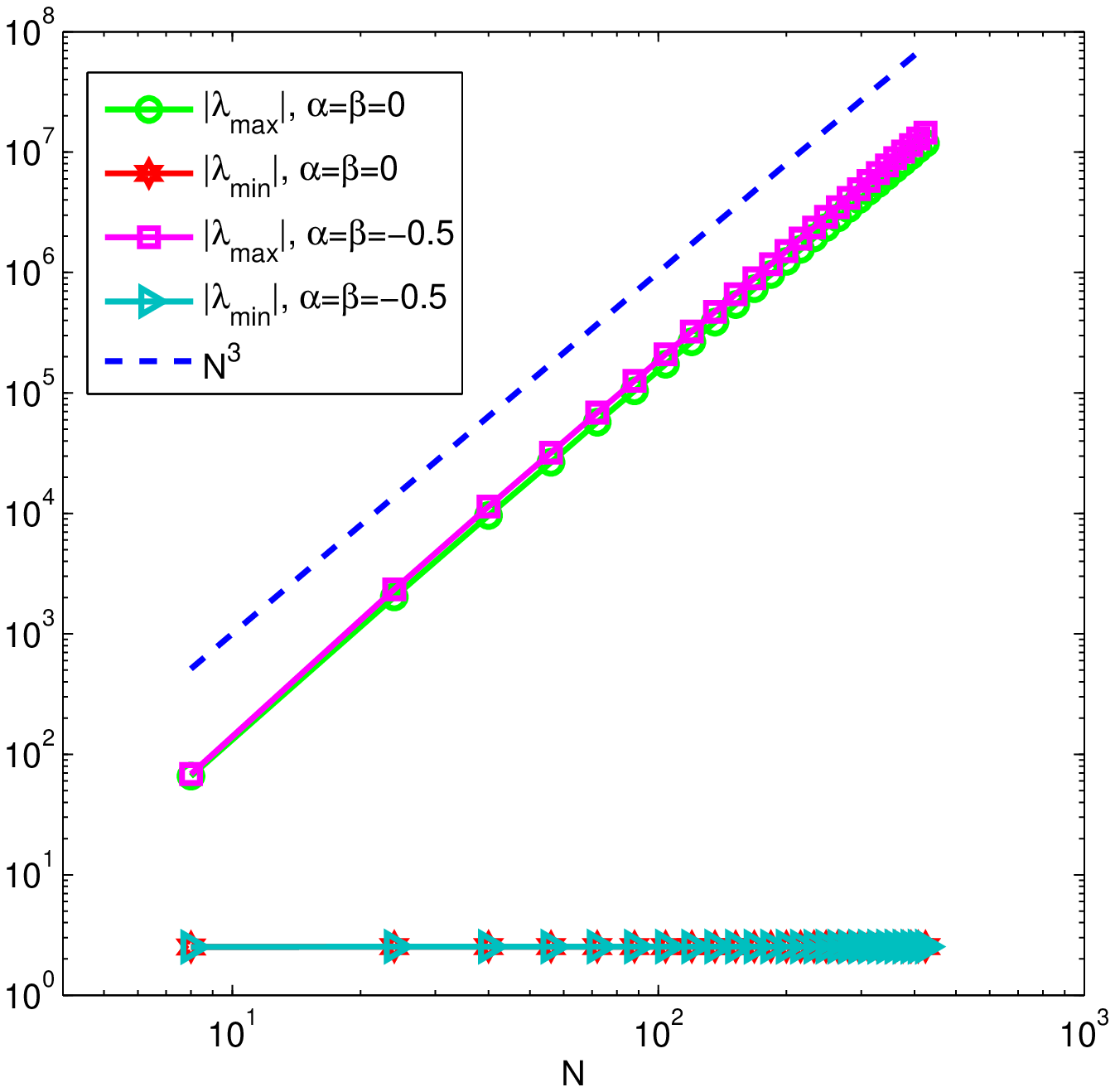}
     \includegraphics[width=0.48\textwidth]{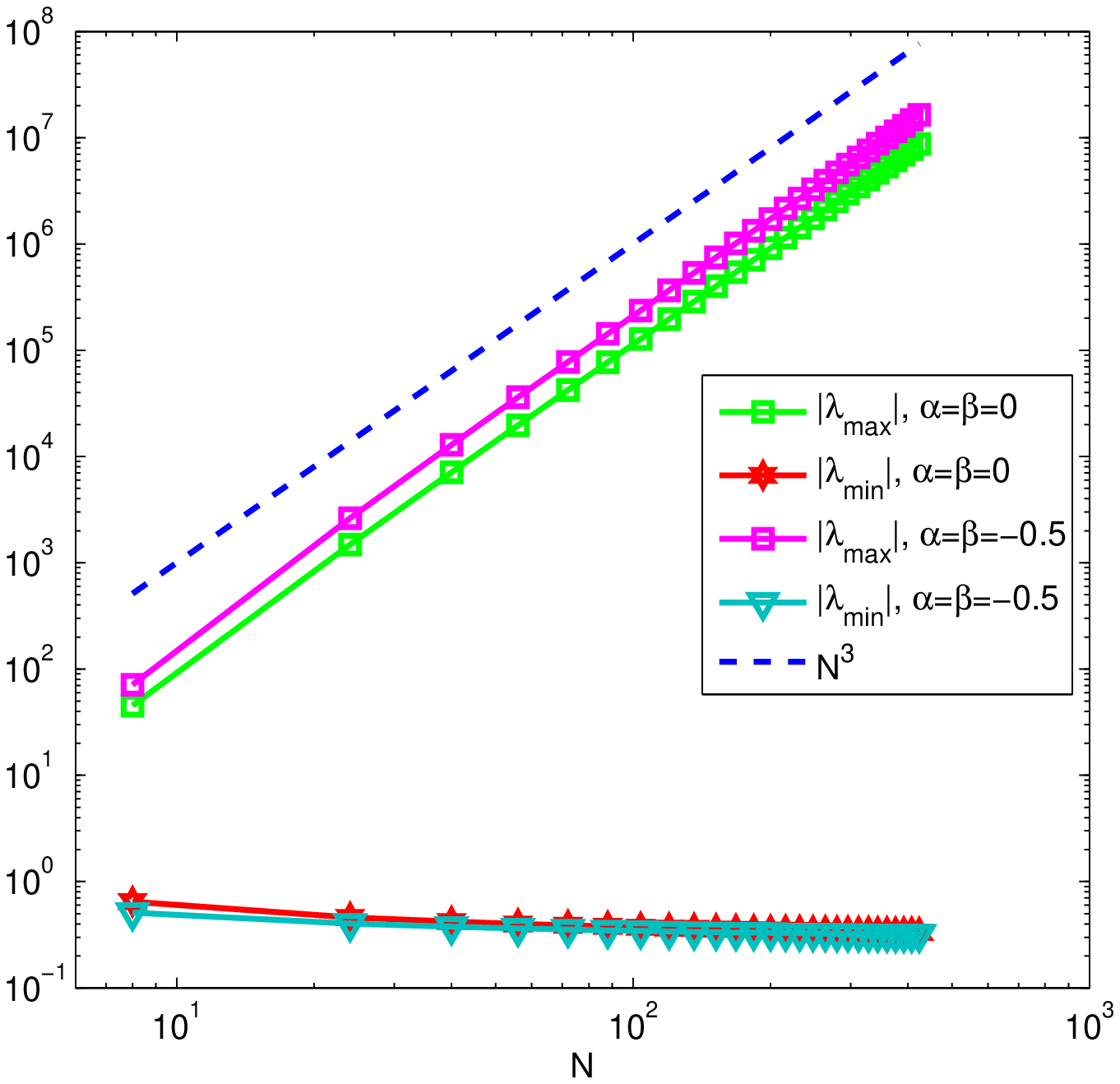}
   \caption{Maximum and minimum (in modulus) eigenvalues   of F-PSDM with $\mu=1.5$. Left:  Caputo. Right: modified Riemann-Liouville.}\label{uhatvsuNhat00}
\end{figure}

\section{Caputo fractional Birkhoff interpolation and   inverse   F-PSDM}\label{sect:caputoinverse}
\setcounter{equation}{0}
\setcounter{lmm}{0}
\setcounter{thm}{0}

As already mentioned,   the condition number of the collocation system of a FDE of order $\mu$ grows  like  $O(N^{2\mu}),$ 
so its solution  suffers from severe round-off errors, and it also becomes rather prohibitive to solve the linear system by an iterative method.  Following the spirit of \cite{CoL10,Wan.SZ14},   we introduce the {\em Caputo fractional 
Birkhoff  interpolation} that generates a new interpolating polynomial basis with remarkable properties: 
\begin{itemize} 
\item[(i)] It provides a stable way to invert the Caputo F-PSDM  in \eqref{submatix}, leading to optimal fractional integration preconditioners for  the ill-conditioned collocation schemes.   
\item[(ii)]  It  offers  a  basis for constructing well-conditioned  collocation schemes.
\end{itemize}

\subsection{Caputo fractional Birkhoff  interpolation} Let $\big\{x_j=x_{N,j}^{(\af,\bt)}\big\}_{j=0}^N$ (with $x_0=-1$ and $x_N=1$) be the  JGL points as before.  Consider the following two interpolating problems: 
\begin{itemize}
\item[(i)] {\em For $\mu\in (0,1)$, the Caputo fractional Birkhoff interpolation is to  find $p\in {\mathcal P}_N $ such that
\begin{equation}\label{BirkhoffFracprob}
{}^C\hspace*{-3pt}{ D}^{\mu}_-\,p(x_j)={}^C\hspace*{-3pt}{ D}^{\mu}_-u(x_j),\;\; 1\le j\le N; \quad p(-1)=u(-1),
\end{equation}
for any $u\in C[-1,1]$ satisfying   ${}^C\hspace*{-3pt}{ D}^{\mu}_- u\in C(-1,1].$  }

 \item[(ii)] {\em  For $\mu\in (1,2)$, the Caputo fractional Birkhoff interpolation is to  find $p\in {\mathcal P}_N $ such that
\begin{equation}\label{BirkhoffFracprob2nd}
{}^C\hspace*{-3pt}{ D}^{\mu}_-\,p(x_j)={}^C\hspace*{-3pt}{ D}^{\mu}_-u(x_j),\;\; 1\le j\le N-1; \quad p(\pm 1)=u(\pm 1),
\end{equation}
for any $u\in C[-1,1]$ satisfying   ${}^C\hspace*{-3pt}{ D}^{\mu}_- u\in C(-1,1).$ } 
\end{itemize}
\begin{rem}\label{Birkrmk} The  usual Birkhoff interpolation is comprehensively  studied  in e.g., 
the monograph \cite{BirkhoffBk}.    Typically, a polynomial Birkhoff interpolation requires  at least one point at which the function and the derivative values are not interpolated consecutively.    For example,  consider a three-point interpolation problem: find $p\in {\mathcal P}_2$  such that
$$p(-1)=u(-1),\quad p'(0)=u'(0),\quad p(1)=u(1).$$
It defines a Birkhoff interpolation problem, since the  function value at $x=0$ is not interpolated, as opposite to the Hermite interpolation.  
Due to the involvement of Caputo fractional derivatives,
we call \eqref{BirkhoffFracprob}  and \eqref{BirkhoffFracprob2nd}  the Caputo  fractional Birkhoff interpolation problems. 
 \qed
\end{rem}

As with the Lagrange interpolation, we search for a nodal basis to represent the interpolating polynomial  $p$.
More precisely,   we look for $Q_j^{\mu}\in  {\mathcal P}_N$  such that  
\begin{itemize}
\item[(i)] for $\mu\in (0,1),$
\begin{equation}\label{BirkhoffFrac}
{}^C\hspace*{-3pt}{ D}^{\mu}_-\,Q_j^\mu(x_i)=\delta_{ij},\;\;\;  1\le i\le N;\quad Q_j^{\mu}(-1)=0,\quad 1\le j\le N,
\end{equation}
with  $Q_0^\mu=1;$
\vskip 6pt 
\item[(ii)] for $\mu\in (1,2),$  
\begin{equation}\label{BirkhoffFrac2}
{}^C\hspace*{-3pt}{ D}^{\mu}_-\,Q_j^\mu(x_i)=\delta_{ij},\;\;\;  1\le i\le N-1;\quad Q_j^\mu(\pm 1)=0,\quad 1\le j\le N-1,
\end{equation}
with  $Q_0^\mu(x)=(1-x)/2$ and $Q_N^\mu(x)=(1+x)/2.$
\end{itemize}
Then, we can express the Caputo fractional Birkhoff  interpolating polynomial $p$  of \eqref{BirkhoffFracprob} and \eqref{BirkhoffFracprob2nd}, respectively, as
\begin{equation}\label{Birkintep}
p(x)=u(-1)+\sum_{j=1}^N  {}^C\hspace*{-3pt}{ D}^{\mu}_-u(x_j)\, Q_j^\mu (x),\quad \mu\in (0,1),
\end{equation}
and
\begin{equation}\label{Birkintep2nd}
p(x)=\frac{1-x} 2 u(-1)+\sum_{j=1}^{N-1}  {}^C\hspace*{-3pt}{ D}^{\mu}_-u(x_j)\, Q_j^\mu(x)+\frac{1+x} 2 u(1),\quad \mu\in (1,2).
\end{equation}
Therefore, $\{Q_j^\mu\}$ are dubbed as the {\em Caputo fractional Birkhoff interpolating basis polynomials of order $\mu$}.

Introduce the matrices 
\begin{equation}\label{newmatrix}
{\bs Q}^{(\mu)}=\begin{cases}
\big({\bs Q}^{(\mu)}_{lj}\big)_{1\le l,j\le N},\quad & {\rm if}\;\; \mu\in (0,1),\\[4pt]
\big({\bs Q}^{(\mu)}_{lj}\big)_{1\le l,j\le N-1},\quad & {\rm if}\;\; \mu\in (1,2),
\end{cases}\quad{\rm where}\;\;\;  {\bs Q}^{(\mu)}_{lj}=Q_j^\mu(x_l).
\end{equation}
Remarkably, the matrix ${\bs Q}^{(\mu)}$ is the inverse of  ${}^C\hspace*{-3pt}{\bs D}_{\rm in}^{(\mu)}$ defined  in \eqref{submatix}.
\begin{thm}\label{Th:inverse} For $\mu\in (k-1,k)$ with $k=1,2,$  we have
\begin{equation}\label{BDBD}
{\bs Q}^{(\mu)}\,  {}^C\hspace*{-3pt}{\bs D}_{\rm in}^{(\mu)}  ={}^C\hspace*{-3pt}{\bs D}_{\rm in}^{(\mu)}\, {\bs Q}^{(\mu)}={\bs I}_{N+1-k},
\end{equation}
where ${\bs I}_{N+1-k}$ is the  identity matrix of order $N+1-k.$
\end{thm}
\begin{proof} We just prove \eqref{BDBD} with  $\mu\in (0,1),$ as the case $\mu\in(1,2)$ can be shown in a similar fashion. 
Since $Q_j^\mu\in {\mathcal P}_N$ and $Q_j^\mu (-1)=0$ for $1\le j\le N,$  we can write
\begin{equation*}\label{bpsrela}
Q_j^\mu (x)=\sum_{l=0}^N Q^\mu_j(x_l) h_l(x)=\sum_{l=1}^N Q^\mu_j(x_l) h_l(x),\;\;\; 1\le j\le N,
\end{equation*}
where $\{h_l\}$ are the Lagrange interpolating  basis polynomials associated with    JGL points.
Thus,
\begin{equation*}\label{bpsrela2}
{}^C\hspace*{-3pt}{ D}^{\mu}_- \,Q^\mu_j(x)=\sum_{l=1}^N Q^\mu_j(x_l) {}^C\hspace*{-3pt}{ D}^{\mu}_-\, h_l(x).
\end{equation*}
Taking $x=x_i$ for $1\le i\le N$ in the above equation,  we obtain \eqref{BDBD} with $\mu\in (0,1)$ from \eqref{BirkhoffFrac} straightforwardly. 
\end{proof}

\subsection{Computing the new basis $\{Q^\mu_j\}$} 

The following property plays a  crucial role in  computing  the new basis $\{Q^\mu_j\},$ which follows from Lemma  \ref{newsolut}. 
\begin{lmm}\label{import form}  Let $\{x_j\}_{j=0}^N$  be the JGL points with $x_0=-1$ and $x_N=1$. Then  for $\mu\in (k-1,k)$ with $k=1,2,$  we have 
\begin{equation}\label{newidentityA}
D^k Q^\mu_j(x)={}^R\hspace*{-2pt}{ D}^{k-\mu}_-\,\Big\{\Big(\frac{1+x}{1+x_j}\Big)^{k-\mu} \hbar_j(x)\Big\},\;\;\; 
1\le j\le N+1-k,
\end{equation}
where  $\{\hbar_j\}_{j=1}^{N+1-k}$ are the Lagrange-Gauss interpolating basis polynomials associated with the JGL points  $\{x_j\}_{j=1}^{N+1-k},$ that is, 
 \begin{equation}\label{hbari}
 \hbar_j\in {\mathcal P}_{N-k}, \quad \hbar_j(x_i)=\delta_{ij}\;\;\; {\rm  for}\;\;\;  1\le i,j\le N+1-k.
 \end{equation} 
\end{lmm}
\begin{proof} Since $Q_j^\mu \in {\mathcal P}_N,$ we obtain  from Lemma  \ref{newsolut} that 
${}^C\hspace*{-3pt}{ D}^{\mu}_-\,Q^\mu_j\in {\mathcal F}_{N-k}^{(k-\mu)}.$
Noting that $\{\hbar_j\}_{j=1}^{N+1-k}$  forms a basis of ${\mathcal P}_{N-k},$ so by \eqref{fspsp}, 
\begin{equation*}
{}^C\hspace*{-3pt}{ D}^{\mu}_-\,Q^\mu_j(x)=\sum_{l=1}^{N+1-k} c_{lj}\; (1+x)^{k-\mu}\hbar_l(x),\quad 1\le j\le N+1-k.
\end{equation*}
Letting $x=x_i$ and using the interpolating conditions, we find that 
$c_{lj}=(1+x_l)^{\mu-k} \delta_{lj}.$ Thus, we obtain 
\begin{equation}\label{newidentity}
{}^C\hspace*{-3pt}{ D}^{\mu}_-\,Q^\mu_j(x)=\Big(\frac{1+x}{1+x_j}\Big)^{k-\mu} \hbar_j(x),\;\;\; 
1\le j\le N+1-k.
\end{equation}
 By the definition \eqref{left-fra-der-c}, we have ${}^C\hspace*{-3pt}{ D}^{\mu}_-=I_-^{k-\mu}D^k,$ so using  \eqref{rulesa}, we obtain  \eqref{newidentityA}  from \eqref{newidentity} immediately. 
\end{proof}

With the aid of \eqref{newidentityA},   we are able to derive the explicit formulas for computing the new basis.  We provide the derivation  in Appendix \ref{AppendixC}. 
\begin{thm}\label{newJacobimat}  Let $\big\{x_j=x_{N,j}^{(\alpha,\beta)}, \omega_j=\omega_{N,j}^{(\alpha,\beta)}\big\}_{j=0}^N$ with {\rm(}$\alpha,\bt>-1$ and $x_0=-x_N=-1${\rm)} be the JGL   quadrature nodes and weights.    Then 
$\{Q_j^\mu\}$ can be computed by 
\begin{itemize}
\item[(i)] For $\mu\in (0,1),$ 
\begin{equation}\label{dQj}
Q_j^\mu(x)=\frac{1}{(1+x_j)^{1-\mu}}\sum_{l=0}^{N-1} \frac{\Gamma(l-\mu+2)}{l!}\, \breve\xi_{lj}\, \int_{-1}^x P_l(x) dx, \quad 1\le j\le N, 
\end{equation}
where 
\begin{align}
& \breve \xi_{lj}=\sum_{n=l}^{N-1} {}^{(\af,\bt)\!}\bs C_{ln}^{(\mu-1,1-\mu)} \xi_{nj}, \label{xilj0}\\
& \xi_{nj}=\frac{1}{\gamma_n^{(\af,\bt)}} \bigg\{-\frac{c_j}{\beta+1} \frac{P_{N}^{(\af,\bt)}(-1)}{P_N^{(\af,\bt)}(x_j)} P_{n}^{(\af,\bt)}(-1) \omega_0
+P_{n}^{(\af,\bt)}(x_j)\omega_j\bigg\}, \label{hjaj10}
\end{align}
with $c_j=1$ for $1\le j\le N-1,$ and $c_N=\alpha+1.$
\item[(ii)] For $\mu\in (1,2),$ 
\begin{equation}\label{dQj2}
Q^\mu_{j}(x)= \frac{1}{(1+x_j)^{2-\mu}} \sum^{N-2}_{l=0}\frac{\Gamma(l-\mu+3)}{l!}
\; \breve \xi_{lj}\, \Phi_{l}(x),\;\;\; 1\le j\le N-1, 
\end{equation}
where 
\begin{align}
&\breve \xi_{lj}=\sum_{n=l}^{N-2} {}^{(\af,\bt)\!}\bs C_{ln}^{(\mu-2,2-\mu)} \xi_{nj}, \label{xilj02}\\
&\xi_{nj}=\frac{1}{\gamma_n^{(\af,\bt)}} \bigg\{\frac{(x_j-1)P_N^{(\af,\bt)}(-1)}{2(\bt+1)P_N^{(\af,\bt)}(x_j)} P_{n}^{(\af,\bt)}(-1)\omega_0\nonumber\\
&\qquad\;\;  -\frac{(1+x_j)P_N^{(\af,\bt)}(1)}{2(\af+1)P_N^{(\af,\bt)}(x_j)} P_{n}^{(\af,\bt)}(1) \omega_N
+P_{n}^{(\af,\bt)}(x_j)\omega_j\bigg\}, \label{hjaj102}
\end{align}
and 
\begin{equation}\label{Phisbas}  
\Phi_l(x):=\frac{1+x} 2\int_{-1}^1 (t-1)P_l(t)\,dt+\int_{-1}^x (x-t)P_l(t)\,dt.
\end{equation}
\end{itemize}
Here,  $\big\{{}^{(\af,\bt)\!}\bs C_{ln}^{(\mu-k,k-\mu)}\big\}$ are the  connection coefficients as defined in 
\eqref{constAB}.
\end{thm}
\begin{rem}\label{imptrem} Observe from \eqref{xilj0} and \eqref{xilj02} 
that  if we take 
$(\alpha,\beta)=(\mu-k,k-\mu)$ with $k=1,2,$ we have $\breve \xi_{lj}=\xi_{lj},$ so \eqref{dQj} and \eqref{dQj2} have the simplest form.  
Thus, it is preferable to choose these special parameters. \qed  
\end{rem}


\section{Modifed RL fractional Birkhoff interpolation and   inverse  F-PSDM}\label{sect:mRL}
\setcounter{equation}{0}
\setcounter{lmm}{0}
\setcounter{thm}{0}

We introduce in this section the fractional Birkhoff interpolation involving  modified 
Riemann-Liouville (RL)  fractional derivatives which offers new polynomial bases for well-conditioned collocation methods for solving FDEs with Riemann-Liouville fractional derivatives. Moreover, we are able to stably compute 
 the inverse matrix of ${}^R\hspace*{-2pt}\widehat{\bs D}_{{\rm in}}^{(\mu)}$ defined in \eqref{submatixRL}. However, this process appears  more involved than the Caputo case in particular for  $\mu\in(1,2).$

\subsection{Modified Riemann-Liouville fractional Birkhoff interpolation}
Like the Caputo case, we consider the modified Riemann-Liouville fractional Birkhoff interpolating problems (i)-(ii) as defined in \eqref{BirkhoffFracprob}-\eqref{BirkhoffFracprob2nd}  with  
${}^R\hspace*{-2pt}\widehat{\bs D}_-^{\mu}$ in place of ${}^C\hspace*{-2pt}{\bs D}_-^{\mu}.$ 
Similarly, we  look for the interpolating basis polynomials $\{\widehat Q_j^{\mu}\}_{j=0}^N\subseteq {\mathcal P}_N$  such that  
\begin{itemize}
\item[(i)] for $\mu\in (0,1),$
\begin{equation}\label{BirkhoffFracRL}
\begin{split}
&  {}^R\hspace*{-2pt}\widehat{ D}_-^{\mu}\,\widehat Q_0^\mu(x_i)=0,\;\;\;\;\;  1\le i\le N;\quad \widehat Q_0^{\mu}(-1)=1,\\
& {}^R\hspace*{-2pt}\widehat{ D}_-^{\mu}\,\widehat Q_j^\mu(x_i)=\delta_{ij},\;\;\;  1\le i\le N,\quad \widehat Q_j^{\mu}(-1)=0,\quad 1\le j\le N;
\end{split}
\end{equation}
\vskip 1pt 
\item[(ii)] for $\mu\in (1,2),$  
\begin{equation}\label{BirkhoffFracRL2}
\begin{split}
&  {}^R\hspace*{-2pt}\widehat{ D}_-^{\mu}\,\widehat Q_0^\mu(x_i)=0,\;\;\;\;\;  1\le i\le N-1;\quad \widehat Q_0^{\mu}(-1)=1,\quad \widehat Q_0^{\mu}(1)=0,\\
& {}^R\hspace*{-2pt}\widehat{ D}_-^{\mu}\,\widehat Q_j^\mu(x_i)=\delta_{ij},\;\;\;  1\le i\le N-1,\quad \widehat Q_j^{\mu}(\pm 1)=0,\quad 1\le j\le N-1,\\
&  {}^R\hspace*{-2pt}\widehat{ D}_-^{\mu}\,\widehat Q_N^\mu(x_i)=0,\;\;\;\;\;  1\le i\le N-1;\quad \widehat Q_N^{\mu}(-1)=0,\quad \widehat Q_N^{\mu}(1)=1.
\end{split}
\end{equation}
\end{itemize}
Then for any $u\in {\mathcal P}_N,$ we can write 
\begin{equation}\label{RLBirkform}
\begin{split}
u(x)&=u(-1) \widehat Q_0^{\mu}(x)+\sum_{j=1}^N {}^R\hspace*{-2pt}\widehat{D}_-^{\mu} u(x_j) \,\widehat Q_j^\mu(x) \qquad \text{(for $\mu\in (0,1)$)}\\
&=u(-1) \widehat Q_0^{\mu}(x)+\sum_{j=1}^{N-1} {}^R\hspace*{-2pt}\widehat{D}_-^{\mu} u(x_j) \,\widehat Q_j^\mu(x)+ u(1) \widehat Q_N^{\mu}(x) \quad \text{(for $\mu\in (1,2)$)}.
\end{split}
\end{equation}

Introduce the matrices generated from the new basis: 
\begin{equation}\label{newmatrixRL}
 \widehat {\bs Q}^{(\mu)}=\begin{cases}
\big( \widehat {\bs Q}^{(\mu)}_{lj}\big)_{1\le l,j\le N},\quad & {\rm if}\;\; \mu\in (0,1),\\[4pt]
\big( \widehat {\bs Q}^{(\mu)}_{lj}\big)_{1\le l,j\le N-1},\quad & {\rm if}\;\; \mu\in (1,2),
\end{cases}\quad{\rm where}\;\;\;   \widehat{\bs Q}^{(\mu)}_{lj}=\widehat{Q}_j^\mu(x_l).
\end{equation}
Like Theorem \ref{Th:inverse}, we can claim that  $\widehat {\bs Q}^{(\mu)}$ is the inverse of ${}^R\hspace*{-2pt}\widehat{\bs D}_{{\rm in}}^{(\mu)}.$  As the proof of the theorem below is very similar to that of Theorem  \ref{Th:inverse},  we omit it.  
\begin{thm}\label{Th:inverseRL}  For $\mu \in(k-1,k)$ with $k=1,2,$ we have
\begin{equation}\label{dQjRL}
\widehat{\bs Q}^{(\mu)}{}^R\hspace*{-2pt}{\widehat{\bs D}}^{(\mu)}_{\rm in}={}^R\hspace*{-2pt}{\widehat{\bs D}}^{(\mu)}_{\rm in} \,\widehat{\bs Q}^{(\mu)}=\bs I_{N+1-k},
\end{equation}
where  $\bs I_{N+1-k}$ is the identity matrix of order $N+1-k$.
\end{thm}

\subsection{Computing the new basis $\big\{\widehat Q^\mu_j\big\}_{j=0}^N$}

The following lemma is  very  useful for the  computation, whose proof is provided in Appendix \ref{AppendixC0}. 
\begin{lmm}\label{zerolmm} 
Let  $\mu\in (k-1,k)$ with $k=1,2.$  Then for  any  $f\in {}_0{\mathcal P}_N,$  the fractional equation
 \begin{equation}\label{newsolu}
 {^R}\hspace*{-2pt} {\widehat D}^{\mu}_-u(x)=f(x),\quad  u(-1)=0, 
 \end{equation}
 has a unique solution $u\in {}_0{\mathcal P}_N$ of the form 
 \begin{equation}\label{soluform}
 u(x)=I_-^\mu \big\{(1+x)^{-\mu} f(x)\big\}.
 \end{equation}
 In particular,  for any $u\in {\mathcal P}_N,$  we have 
 \begin{equation}\label{Rrhoeqn}
 {^R}\hspace*{-2pt} {\widehat D}^{\mu}_-u(-1)=0  \;\;\;\text{if and only if}\;\;\; u(-1)=0.
 \end{equation}
\end{lmm}

For clarity of presentation, we   deal with two cases:  (i) $\mu\in (0,1)$ and (ii)  $\mu\in (1,2),$ separately. 
\subsubsection{$\big\{\widehat Q^\mu_j\big\}_{j=0}^N$ with $\mu\in (0,1)$} Using  the properties \eqref{newpropC} and \eqref{Rrhoeqn}, we obtain  from the interpolating conditions in \eqref{BirkhoffFracRL} that 
\begin{equation}\label{identiyRLj}
\big({}^R\hspace*{-2pt}\widehat{D}_-^{\mu}\,\widehat Q_j^\mu\big)(x)=h_j(x), \;\;\; 1\le j\le N;\quad 
\big({}^R\hspace*{-2pt}\widehat{ D}_-^{\mu}\,\widehat Q_0^\mu\big)(x)=\xi h_0(x),
\end{equation}
where $\{h_j\}$ are the JGL interpolating basis polynomials defined in \eqref{jacbasisfun}, and $\xi$ is a constant to be determined by $\widehat Q_0^\mu (-1)=1.$   Note that thanks to \eqref{Rrhoeqn}, the condition  
$\widehat Q_j^\mu(-1)=0$ is built-in, as $h_j(-1)=0$ for $1\le j\le N.$ 
We summarise below the explicit representation  of the new basis.   Once again, we put the proof in Appendix \ref{Appendixnew}.
\begin{thm}\label{invRLmu01}  Let $\big\{x_j=x_{N,j}^{(\alpha,\beta)}, \omega_j=\omega_{N,j}^{(\alpha,\beta)}\big\}_{j=0}^N $ {\rm(}with $\af,\bt>-1$ and $x_0=-1${\rm)} be the JGL   quadrature points and weights.    Then 
$\big\{\widehat Q_j^\mu\big\}_{j=0}^N$ with $\mu\in (0,1)$ can be computed by 
\begin{equation}\label{qjmu01}
\widehat Q_j^\mu(x)=  \zeta_j\sum_{l=0}^N  \frac{\Gamma(l-\mu+1)} {l!}\, \hat t_{lj}\, P_l(x)\;\;\; 
{\rm with}\;\;\; \hat t_{lj}=
\sum_{n=l}^N {}^{(\af,\bt)\!}C_{ln}^{(\mu,-\mu)} t_{nj}, 
\end{equation}
where  $\zeta_0=1/\Gamma(1-\mu),$  $\zeta_j=1$ for $1\le j\le N,$ $\big\{{}^{(\af,\bt)\!}C_{ln}^{(\mu,-\mu)}\big\}$ are the connection coefficients defined in 
Subsection {\rm \ref{connprob}}, and  
\begin{equation}\label{constAj}
t_{nj}= \frac{\omega_j}{\tilde \gamma_n^{(\af,\bt)}} P_n^{(\af,\bt)}(x_j),
\end{equation}
with  $\tilde \gamma_n^{(\af,\bt)}$  being defined in \eqref{tildgamma}. 
\end{thm}
\begin{rem}\label{modm0} If $(\af,\bt)=(\mu,-\mu)$ with $\mu\in (0,1),$ we have $\hat t_{lj}=t_{lj},$ so $\widehat Q_j^\mu$ has the simplest form. \qed
\end{rem}

\subsubsection{ $\big\{\widehat Q^\mu_j\big\}_{j=0}^N$ with $\mu\in (1,2)$}   It is essential to derive the  identities like \eqref{identiyRLj}. Indeed,  using  \eqref{newpropC} and \eqref{Rrhoeqn}, we obtain  from the interpolating conditions in \eqref{BirkhoffFracRL2} that  
\begin{align}
& \big({}^R\hspace*{-2pt}\widehat{D}_-^{\mu}\,\widehat Q_0^\mu\big)(x)=(\tau_0+\kappa_0\, x) \hat h_0(x), \quad  \widehat Q_0^\mu(-1)=1,\quad  \widehat Q_0^\mu(1)=0; \label{identiyRL02}\\
& \big({}^R\hspace*{-2pt}\widehat{D}_-^{\mu}\,\widehat Q_j^\mu\big)(x)=\frac{x+\tau_j}{x_j+\tau_j} \hat h_j(x),  \quad \widehat Q_j^\mu(1)=0,\quad 1\le j\le N-1; \label{identiyRLj2}\\
& \big({}^R\hspace*{-2pt}\widehat{D}_-^{\mu}\,\widehat Q_N^\mu\big)(x)=\tau_N (1+x)\, \hat h_0(x),
\quad  \widehat Q_N^\mu(1)=1, \label{identiyRLN2}
\end{align}
where $\{\hat h_j\}_{j=0}^{N-1}$ are the Lagrange interpolating basis polynomials at  JGL points $\{x_j\}_{j=0}^{N-1},$ that is, 
\begin{equation}\label{hathj}
\hat h_j(x)\in {\mathcal  P}_{N-1},\quad \hat h_j(x_i)=\delta_{ij},\quad 0\le i,j\le N-1. 
\end{equation}
 In \eqref{identiyRL02}-\eqref{identiyRLN2}, 
 $\{\tau_j\}_{j=0}^{N}$ and $\kappa_0$ are constants to be determined by the corresponding conditions at $x=\pm 1$, e.g., $\widehat Q_j^\mu(1)=0$ in \eqref{identiyRLj2}. 
 It is noteworthy  that thanks to  \eqref{Rrhoeqn},  the interpolating condition: $\widehat Q_j^{\mu}(-1)=0$ is built in $\big({}^R\hspace*{-2pt}\widehat{D}_-^{\mu}\,\widehat Q_j^\mu\big)(-1)=0$ for $1\le j\le N.$

 In what follows, we shall use  the three-term  recurrence relation of Jacobi polynomials (cf.   \cite[(3.110)]{ShenTangWang2011}):
 \begin{equation}\label{ThreeTermExpress}
xP^{(\mu,1-\mu)}_{l}(x)=a_{l+1} P^{(\mu,1-\mu)}_{l+1}(x)+b_{l}P^{(\mu,1-\mu)}_{l}(x)+c_{l-1}P^{(\mu,1-\mu)}_{l-1}(x), \quad  l\geq 0,
\end{equation}
where $c_{-1}=0,$  $\mu\in (1,2),$ and 
\begin{equation}\label{3termcoef}
a_{l+1}=\displaystyle\frac{l+2}{2l+3},\quad b_l=\displaystyle\frac{1-2\mu}{(2l+1)(2l+3)},\quad c_{l-1}=\displaystyle\frac{(l+\mu)(l-\mu+1)}{(l+1)(2l+1)}.
\end{equation}
 
As before, it is necessary to expand $\{\hat h_j\}$ in terms of Jacobi polynomials with different parameters by using the notion of connection problems, so as to use compact  and closed-form formulas to compute the new basis.  We state below the connections of three expansions, and postpone the derivations in Appendix \ref{AppendixD}.  
 \begin{lmm}\label{Jcbiexplmm} Let $\big\{x_j=x_{N,j}^{(\alpha,\beta)}, \omega_j=\omega_{N,j}^{(\alpha,\beta)}\big\}_{j=0}^N $ {\rm(}with $\af,\bt>-1$ and $x_0=-x_N=-1${\rm)} be the JGL   quadrature nodes and weights, and let $\{\hat h_j\}_{j=0}^{N-1}$ be the Lagrange interpolating basis polynomials associated with  $\{x_j\}_{j=0}^{N-1}$ defined in \eqref{hathj}.  Then for $\mu\in (1,2),$  we have 
   \begin{equation}\label{hathjexp}
   \begin{split}
  \hat h_j(x)&=\sum_{n=0}^{N-1}\varrho_{nj}\, P_n^{(\af,\bt)}(x)=\sum_{l=0}^{N-1}\tilde \varrho_{lj}\, P_l^{(\mu,1-\mu)}(x)\\
  &=\hat \varrho_{0j}+\sum_{l=0}^{N-2}\hat \varrho_{l+1,j}\,(1+x)P_{l}^{(\mu,1-\mu)}(x),\;\; 0\le j\le N-1, 
  \end{split}
  \end{equation} 
  where $\hat \varrho_{00}=1$ and $ \hat \varrho_{0j}=0$   for $1\le j\le N-1.$ Moreover,  the coefficients can be computed by 
    \begin{align}
  &\varrho_{n0}=\frac{1}{\gamma_n^{(\af,\bt)}} \Big\{P_{n}^{(\af,\bt)}(-1)\omega_0-\frac{\beta+1}{\alpha+1} \frac{P_N^{(\af,\bt)}(1)} {P_N^{(\af,\bt)}(-1)} P_{n}^{(\af,\bt)}(1)\omega_N\Big\}, \label{newconstA}\\
  &\varrho_{nj}=\frac{1}{\gamma_n^{(\af,\bt)}} \Big\{P_{n}^{(\af,\bt)}(x_j)\omega_j-\frac{1}{\alpha+1} \frac{P_N^{(\af,\bt)}(1)} {P_N^{(\af,\bt)}(x_j)} P_{n}^{(\af,\bt)}(1)\omega_N\Big\},\;\;\; 1\le j\le N-1,
   \label{newconstBj}\\
 &  \tilde \varrho_{lj}=  \sum_{n=l}^{N-1} {}^{(\af,\bt)\!}C_{ln}^{(\mu,1-\mu)} \varrho_{nj},\quad 0\le l,j\le N-1,  \label{connectjj}
  \end{align}
and by  the backward recurrence relation: 
\begin{equation}\label{backformula}
\begin{split}
&\hat \varrho_{ij}=\frac 1 {a_i} \tilde \varrho_{ij}-\frac {b_i+1} {a_i} \hat \varrho_{i+1,j} -\frac {c_i} {a_i} \hat \varrho_{i+2,j}, \quad i=N-3,N-2,\cdots, 1,\\
&\hat \varrho_{N-1,j} = \frac 1 {a_{N-1}} \tilde \varrho_{N-1,j},\quad 
\hat \varrho_{N-2,j}=\frac 1 {a_{N-2}} \tilde \varrho_{N-2,j}-\frac {b_{N-2}+1} {a_{N-2}} \hat \varrho_{N-1,j},
\end{split}
\end{equation}
where $\{a_i,b_i,c_i\}$ are given  in \eqref{3termcoef}. 
\end{lmm}

With the above preparations, we are ready to derive the explicit formulas of the new basis 
$\big\{\widehat Q_j^\mu\big\}_{j=0}^N$ with $\mu\in (1,2).$  We refer to  Appedix \ref{AppendixE} for the derivation.  
\begin{thm}\label{invRLmu12}  
Let $\big\{\tilde \varrho_{lj}, \hat\varrho_{lj}\big\}$ be the coefficients defined in  Lemma {\rm \ref{Jcbiexplmm}}, and   denote 
\begin{equation}\label{newconstA2}
d_{l}^\mu:=\frac{\Gamma(l+2-\mu)}{(l+1)!},\quad \varphi_l(x):=(1+x)P_l^{(0,1)}(x).
\end{equation}
 Then 
$\big\{\widehat Q_j^\mu\big\}_{j=0}^N$ at JGL points with $\mu\in (1,2)$ can be computed by 
\begin{itemize}
\item[(i)] for $j=0,$ 
\begin{equation}\label{Q0solunewm}
\begin{split}
\widehat Q_0^\mu(x)=&� 1+\Big(\tau_0 -\frac 1{\Gamma(1-\mu)} \Big)
\sum_{l=0}^{N-1} d_l^\mu
\tilde  \varrho_{l0}\,\varphi_l(x) 
+ \frac 1{\Gamma(1-\mu)} \sum_{l=0}^{N-2} d_l^\mu
\hat  \varrho_{l+1,0}\,\varphi_l(x), \\
{\rm where}\;\;&\;\;\tau_0-\frac 1{ \Gamma(1-\mu)}  =-\bigg\{\frac 1 2+ \frac 1{ \Gamma(1-\mu)} \sum_{l=0}^{N-2} d_l^\mu
\hat  \varrho_{l+1,0}\,\bigg\}\bigg/\sum_{l=0}^{N-1} d_l^\mu
\tilde  \varrho_{l0};
\end{split}
\end{equation} 
\item[(ii)] for $j=N,$
\begin{equation}\label{QNformula}
\widehat Q_N^\mu(x)= \frac  1 2 \sum_{l=0}^{N-1} d_l^\mu
\tilde  \varrho_{l0}\,\varphi_l(x)\bigg/\sum_{l=0}^{N-1} d_l^\mu
\tilde  \varrho_{l0}\,; 
\end{equation} 
\item[(ii)] for $1\le j \le N-1,$
\begin{equation}\label{Qjsoluform}
\begin{split}
&\widehat Q_j^\mu(x)= \frac {\tau_j}{x_j+\tau_j}\sum_{l=0}^{N-2} d_l^\mu 
\hat  \varrho_{l+1,j}\,\varphi_l(x) + \frac {1}{x_j+\tau_j} \\
 &\quad \times \sum_{l=0}^{N-2} d_l^\mu
\hat  \varrho_{l+1,j}\,\bigg\{\frac{l+2-\mu}{l+2} a_{l+1} \varphi_{l+1}(x)+b_l \varphi_{l}(x)+  \frac{l+1}{l+1-\mu} c_{l-1}\varphi_{l-1}(x)\bigg\},\\
 \end{split}
 \end{equation}
where   $\{a_l, b_l, c_l\}$ {\rm(}with $c_{-1}=0{\rm)}$ are defined in \eqref{3termcoef}, and 
\begin{equation}\label{tauj}
 \tau_j=-1+\mu \,\Gamma(2-\mu) \hat \varrho_{1j}\bigg/{\sum_{l=0}^{N-2} d_l^\mu \hat  \varrho_{l+1,j}}. 
 \end{equation}
\end{itemize}
\end{thm}
\begin{rem}\label{modm12}  We see from \eqref{connectjj} that if $(\af,\bt)=(\mu,1-\mu)$ with $\mu\in (1,2),$  the connections coefficients are not involved, so   $\widehat Q_j^\mu$ has simpler form. \qed
\end{rem}

\section{Well-conditioned collocation schemes and numerical results}
\setcounter{equation}{0}
\setcounter{lmm}{0}
\setcounter{thm}{0}

In this section, we apply the tools developed in previous sections  to construct well-conditioned collocation schemes for initial-valued  or boundary-valued FDEs, and provide ample numerical results to  show the accuracy and stability of the methods. 
 
\subsection{Initial-valued Caputo FDEs} To fix the idea, we first consider the Caputo  FDE of order $\mu\in (0,1):$
\begin{equation}\label{CFDE00}
{^C}\hspace*{-3pt}D_-^\mu u(x)+\lambda(x) u(x)=f(x),\quad x\in (-1,1];\quad u(-1)=u_-, 
\end{equation}
where $\lambda, f$ are given continuous functions, and $u_-$  is a given constant.  
The collocation scheme is to find $u_N\in {\mathcal P}_N$ such that 
\begin{equation}\label{COL00}
{^C}\hspace*{-3pt}D_-^\mu u_N(x_j)+\lambda(x_j) u_N(x_j)=f(x_j),\quad 1\le j\le N;\quad u_N(-1)=u_-.
\end{equation}

The corresponding linear  system under the Lagrange basis polynomials  $\{h_j\}$ (L-COL) becomes 
\begin{equation}\label{COLsys00}
\big({}^C\hspace*{-3pt}{\bs D}_{\rm in}^{(\mu)}+\bs \Lambda \big)\bs u=\bs f,
\end{equation}
where ${}^C\hspace*{-3pt}{\bs D}_{\rm in}^{(\mu)}$ is defined as in \eqref{submatix}, 
$\bs \Lambda={\rm diag}(\lambda(x_1),\cdots,\lambda(x_N)),$ and 
\begin{equation}\label{matrixnews}
\bs u=\big(u_N(x_1),\cdots, u_N(x_N)\big)^t,\;\; \bs f=\big(f(x_1)-u_-{^C}\hspace*{-3pt}D_-^\mu h_0(x_1), \cdots, f(x_N)-u_-{^C}\hspace*{-3pt}D_-^\mu h_0(x_N)\big)^t.
\end{equation}

The collocation system under the Birkhoff interpolation basis polynomials  $\{Q_j^\mu\}$ in 
\eqref{BirkhoffFrac}  (B-COL) becomes 
\begin{equation}\label{COLsys11}
\big(\bs I_{N}+\bs \Lambda \bs Q^{(\mu)} \big)\bs v=\bs g,
\end{equation}
where 
\begin{equation}\label{vjeqn}
u_N(x)=u_- +\sum_{i=1}^N v_j Q_j^\mu(x), \quad \bs v=\big(v_1,\cdots,v_N\big)^t, 
\end{equation}
and  $\bs g=\big(f(x_1)-u_-\lambda(x_1),\cdots, f(x_N)-u_-\lambda(x_N)\big)^t$.  
It is noteworthy that different from  \eqref{COLsys00},  the unknowns of \eqref{COLsys11} are not the approximation of $u$ at the collocation points, but of the Caputo fractional derivative values  in view of \eqref{Birkintep}.

Thanks to Theorem \ref{Th:inverse}, we can precondition \eqref{COLsys00} and obtain the PL-COL system:
\begin{equation}\label{COLsys22}
\big(\bs I_{N}+ \bs Q^{(\mu)} \bs \Lambda \big)\bs u=\bs Q^{(\mu)} \bs f. 
\end{equation}

In the computation, we take $\lambda(x)=2+\sin(25x)$ and $u(x)=E_{\mu,1}(-2(1+x)^{\mu})$ with $\mu=0.8$ in \eqref{CFDE00}, where  the Mittag-Leffler function is defined by 
\begin{equation}\label{Mlfunc}
E_{\alpha,\beta}(z)=\sum_{n=0}^\infty \frac{z^n}{\Gamma(n \alpha  +\beta)}. 
\end{equation}
In view of Remark \ref{imptrem},  we choose the JGL points with $(\alpha,\beta)=(\mu-1,1-\mu)=(-0.2,0.2).$ 
We compare the condition numbers,  number of iterations (using  BiCGSTAB in Matlab)  and convergence behaviour (in discrete $L^2$-norm on fine equally-spaced grids)  of three schemes  (see  Figure  \ref{uhatvsuNhat}). 
Observe from Figure \ref{uhatvsuNhat} (left) that the condition number of usual L-COL divided by $N^{2\mu}$ behaves 
like a constant, while that of PL-COL and B-COL remains a constant even for $N$ up to $2000.$ 
As a result, the latter two schemes only require about $8$ iterations to converge, while the usual L-COL scheme requires  much more iterations with a degradation of  accuracy as depicted in Figure \ref{uhatvsuNhat} (middle). We record the convergence history of three methods in Figure \ref{uhatvsuNhat} (right), and observe that two new schemes are stable even for very large $N.$

\begin{figure}[h!]
  ~\hspace*{-16pt}  \includegraphics[width=0.38\textwidth]{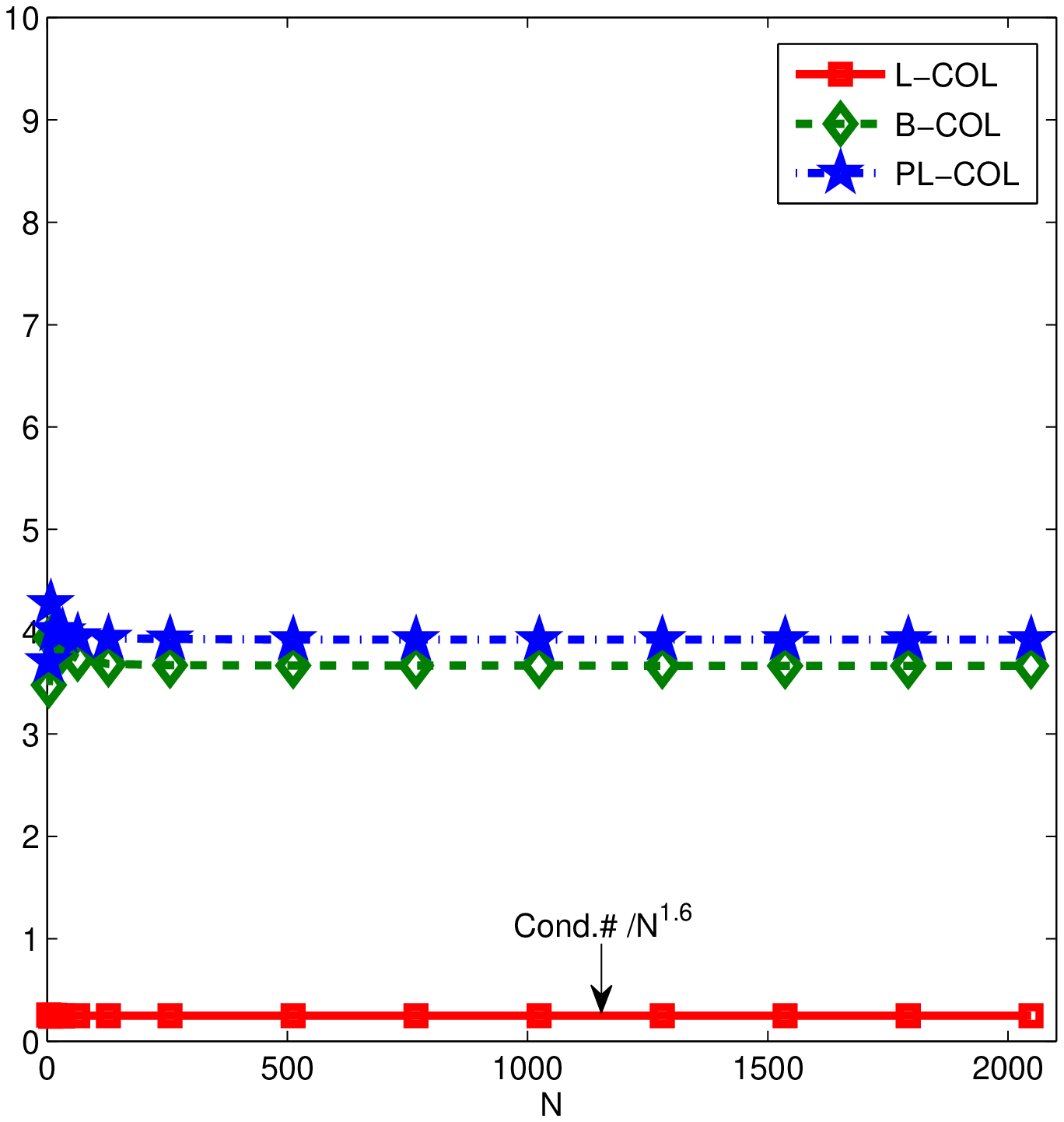}\hspace*{-15pt}
     \includegraphics[width=0.38\textwidth]{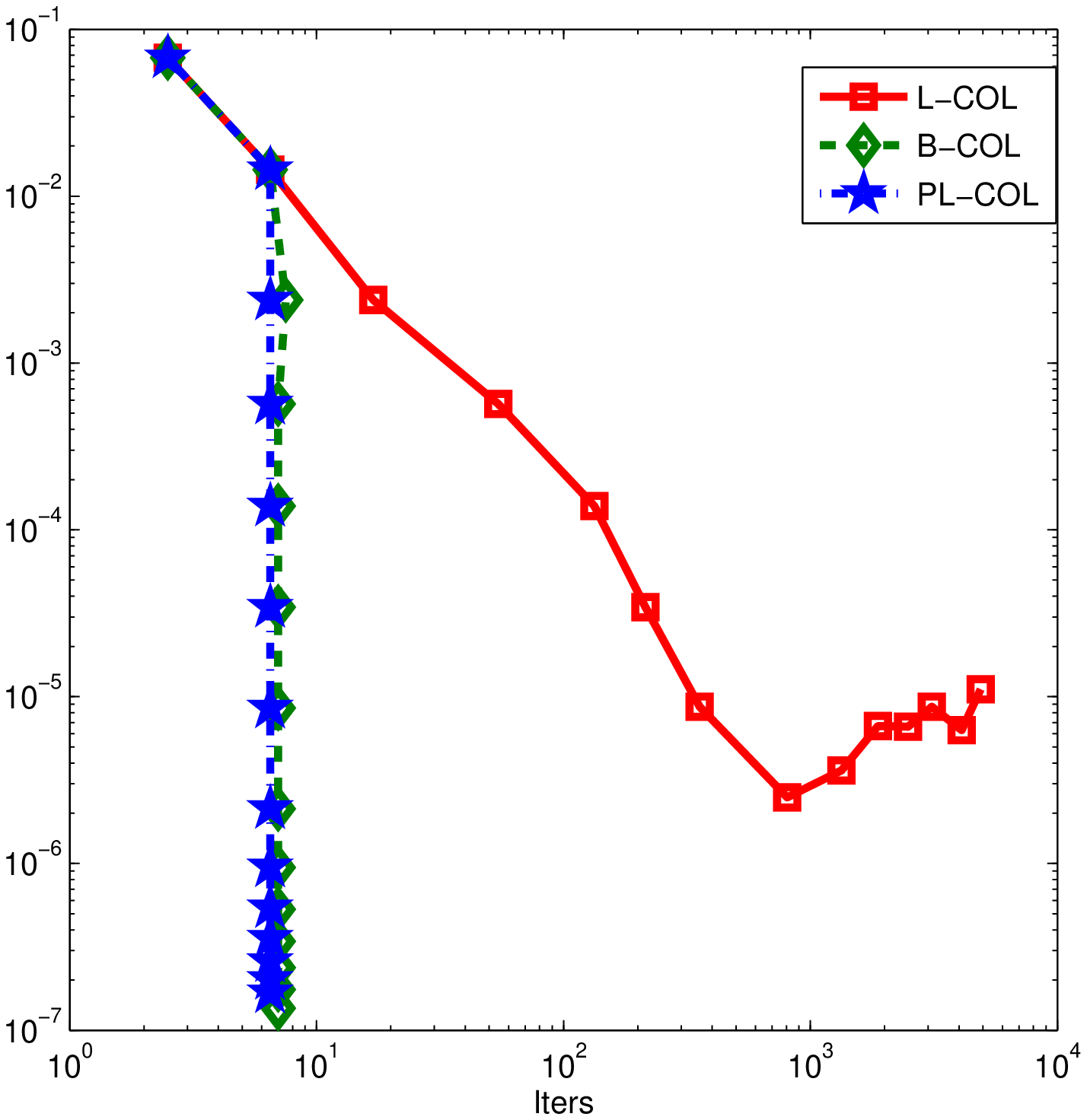}\hspace*{-14pt}
     \includegraphics[width=0.38\textwidth]{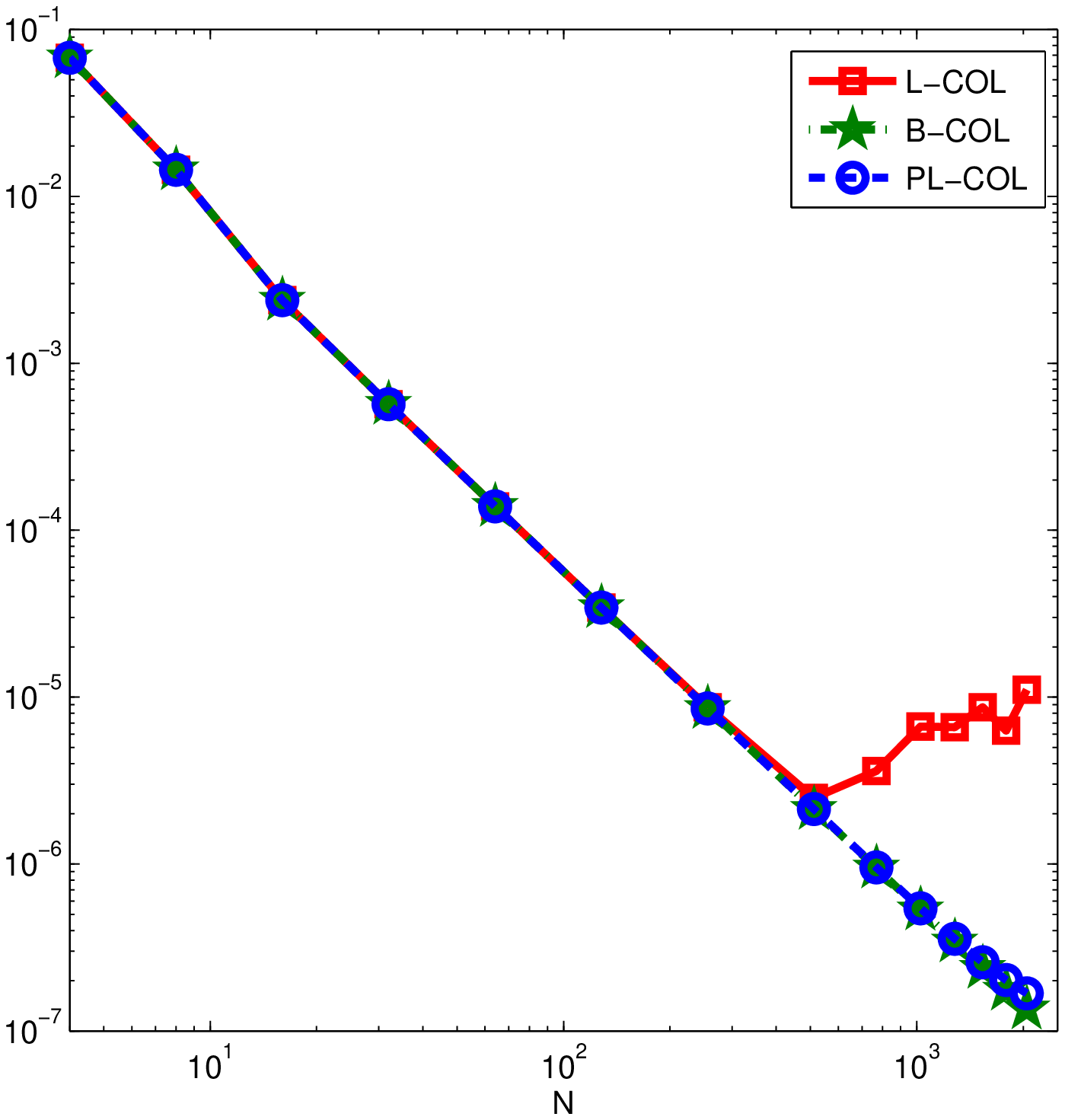}
   \caption{Comparison of condition numbers (left), iteration numbers against errors (middle), and errors against $N$ at convergence  in log-log scale  (right) for  (\ref{CFDE00}). }
    \label{uhatvsuNhat}
\end{figure}

%

\subsection{Boundary-valued Caputo FDEs}   We now turn to  the boundary value problem:
\begin{equation}\label{Bvp00}
\begin{split}
& {^C}\hspace*{-3pt}D_-^\mu u(x)+\lambda_1(x){^C}\hspace*{-3pt}D_-^\nu u(x)+\lambda_2(x) u(x)=f(x),\quad x\in (-1,1);\\
& u(-1)=u_-, \quad u(1)=u_+,\quad  0<\nu<\mu,\;\;\;  \mu\in (1,2),
\end{split}
\end{equation}
where $\lambda_1,\lambda_2$ and $f$ are given continuous functions, and $u_\pm$ are given constants. 

\begin{figure}[h!]
  ~\hspace*{-16pt}  \includegraphics[width=0.38\textwidth]{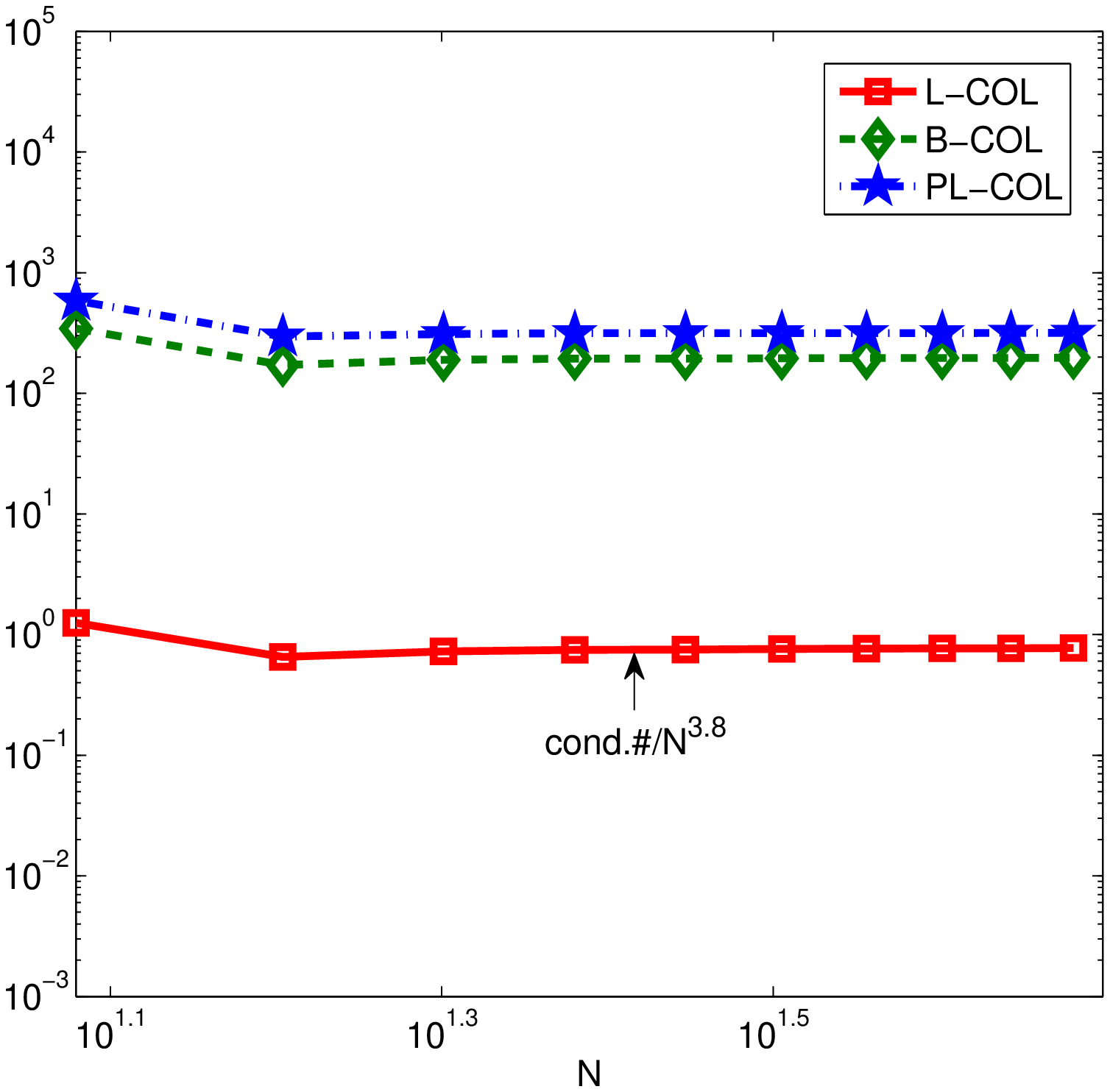}\hspace*{-15pt}
     \includegraphics[width=0.38\textwidth]{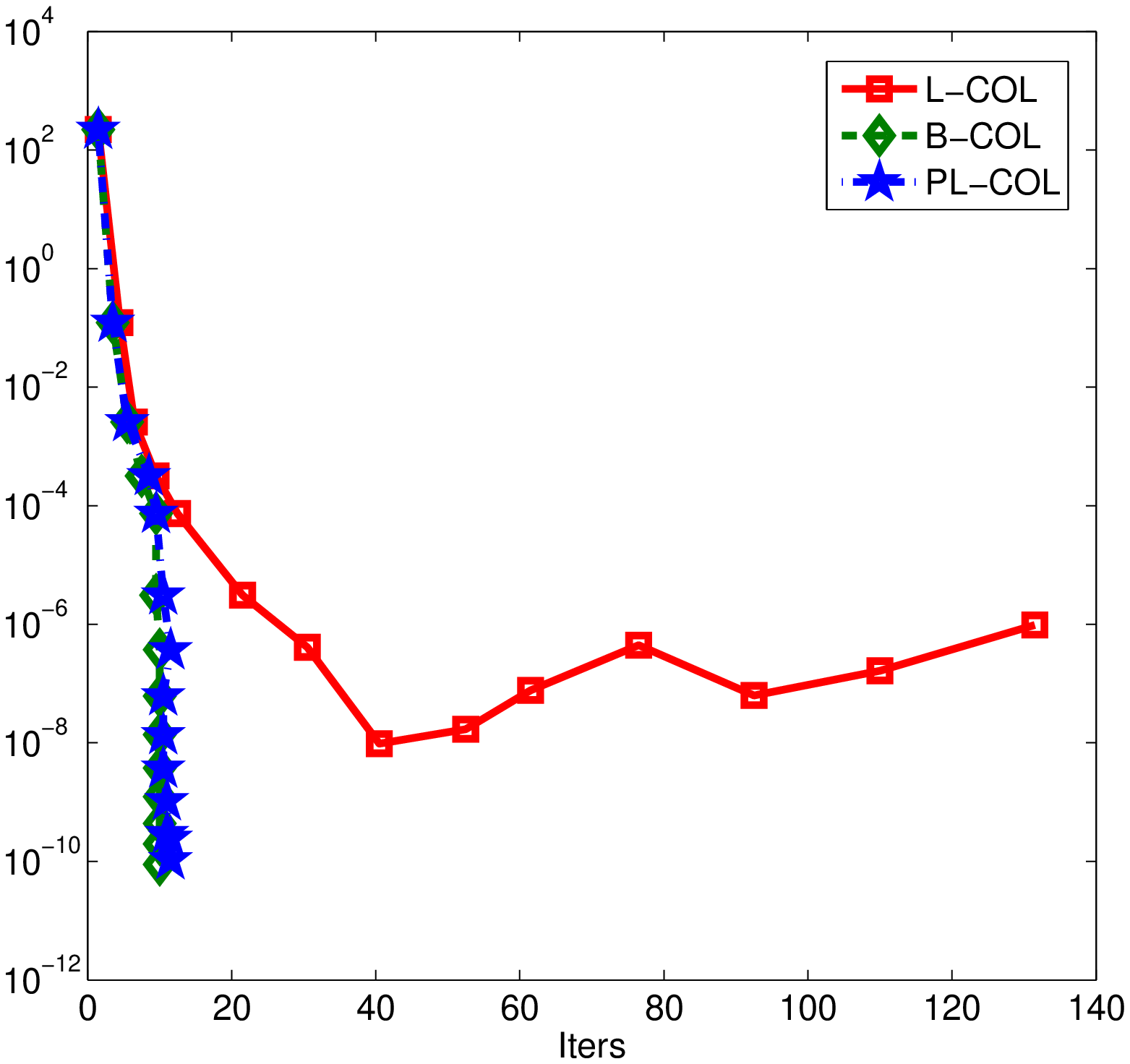}\hspace*{-14pt}
     \includegraphics[width=0.38\textwidth]{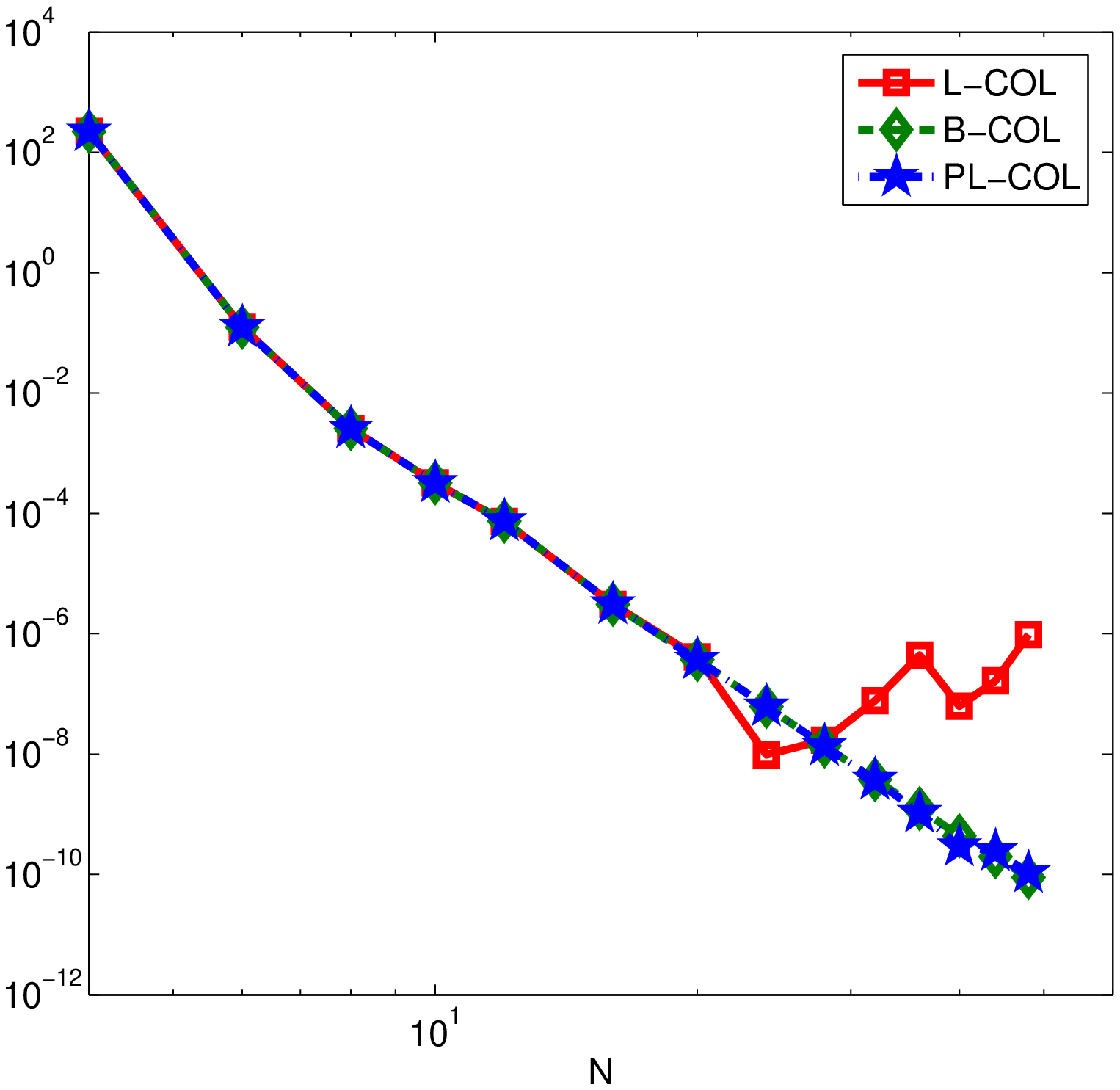}
   \caption{Comparison of condition numbers (left), iteration numbers against errors (middle), and errors against $N$ at convergence  in log-log scale  (right) for  (\ref{Bvp00}).}
    \label{uhatvsuNhatB}
\end{figure}
 
With a pre-computation of the Caputo fractional differentiation matrices of order $\mu$ and $\nu$ in Subsection \ref{sub:matrixDiff},  we can formulate the L-COL scheme as \eqref{COLsys00} straightforwardly.    The counterpart of 
\eqref{COLsys11} i.e., the B-COL scheme, can be formulated as follows: find  (cf. \eqref{Birkintep2nd}) 
\begin{equation}\label{Birkintep2ndss}
u_N(x)=u_N^*(x)+\sum_{j=1}^{N-1}  v_j\, Q_j^\mu(x),\quad \mu\in (1,2); \quad u_N^*(x):=\frac{1-x} 2 u_-+\frac{1+x} 2 u_+, 
\end{equation}
such that 
\begin{equation}\label{COLsys1mu2}
\big(\bs I_{N-1}+\bs \Lambda_1 \bar {\bs Q}^{(\nu)} +\bs \Lambda_2 \bs Q^{(\mu)} \big)\bs v=\bs g,
\end{equation}
where $\bs \Lambda_i={\rm diag}(\lambda_i(x_1),\cdots,\lambda_i(x_{N-1})), i=1,2,$ $\bar {\bs Q}^{(\nu)}_{ij}=
{^C}\hspace*{-3pt}D_-^\nu Q_j^\mu(x_i), 1\le i,j\le N-1,$ and
$\bs g=\big(f(x_1)-q_*(x_1),\cdots, f(x_{N-1})-q_*(x_{N-1})\big)^t$ with  $q_*= \lambda_1{^C}\hspace*{-3pt}D_-^\nu u_N^*+\lambda_2\, u_N^*.$ Note that the entries of  $\bar {\bs Q}^{(\nu)}$ can be evaluated by   Theorem \ref{JacobiForm2}, \eqref{DSc} and  \eqref{dQj2}-\eqref{Phisbas}. Here, we omit  the details.

\begin{rem}\label{justifA}  If $\lambda_1=0$ and $\lambda_2$ is a constant, we can follow \cite[Proposition 3.5]{Wan.SZ14} to justify  the coefficient matrix of \eqref{COLsys1mu2} is well-conditioned.  Indeed, thanks to Theorem \ref{Th:inverse}, the eigenvalues  $\sigma$ of $\bs I_{N-1}-\lambda_2 \bs Q^{(\mu)}$ satisfy 
$$ 1+{\lambda_2} {\lambda_{\rm max}^{-1}} \le \sigma \le 1+  {\lambda_2} {\lambda_{\rm min}^{-1}},$$
where $\lambda_{\rm max }$ and $\lambda_{\rm min}$ are respectively the largest and smallest eigenvalues of ${}^C\hspace*{-3pt}{\bs D}_{\rm in}^{(\mu)}.$  Since $\lambda_{\rm min}=O(1)$ (see Figure \ref{uhatvsuNhat00} (right)), the condition number of $\bs I_{N-1}-\lambda_2 \bs Q^{(\mu)}$  is independent of $N.$ \qed 
\end{rem}

Like \eqref{COLsys22}, we can precondition  the L-COL scheme by ${\bs Q}^{(\mu)}$  which leads  to the PL-COL system.  

In the following comparison, we set $\mu=1.9,  \nu=0.7$ and $(\alpha,\beta)=(-0.1, 0.1)$ (cf.  Remark \ref{imptrem}), and  take 
\begin{equation}\label{lam12}
 \lambda_{1}(x)=2+\sin(4\pi x),\;\;\;\;  \lambda_{2}(x)=2+\cos x,
 \end{equation}
 and 
 \begin{equation}\label{uexct}
 \;\; u(x)=e^{1+x}+(1+x)^{6+{4}/{7}}-2(1+x)^{5+{4}/{7}},
 \end{equation}
 where we can use the formula
 $${^C}\hspace*{-3pt}D_-^\mu e^{1+x}= (1+x)^{k-\mu}E_{1,k+1-\mu}(1+x),\quad \mu\in (k-1,k),\;\; k=1,2, $$
 to work out $f(x).$

Once again, we observe from Figure  \ref{uhatvsuNhatB} that the new schemes: B-COL and PL-COL are well-conditioned, 
attain the expected convergence order about   $10$ iterations, and lead to stable computation for large   $N.$

\subsection{Riemann-Liouville FDEs}  Consider the Riemann-Liouville version of  \eqref{Bvp00}:
\begin{equation}\label{Bvp00A}
\begin{split}
& {^R}\hspace*{-2pt}D_-^\mu u(x)+\lambda_1(x){^R}\hspace*{-2pt}D_-^\nu u(x)+\lambda_2(x) u(x)=f(x),\quad x\in (-1,1);\\
& u(-1)=u_-, \quad u(1)=u_+,\quad  0<\nu<\mu,\;\;\;  \mu\in (1,2),
\end{split}
\end{equation}
where $\lambda_1,\lambda_2$ and $f$ are given continuous functions, and $u_\pm$ are given constants. 

For a better treatment of the singularity,  we consider the modified Riemann-Liouville fractional collocation scheme:
 find $u_N\in {\mathcal P}_N$ such that 
\begin{equation}\label{Bvp01}
\begin{split}
& {^R}\hspace*{-2pt}\widehat{D}_-^\mu u_N(x_j)+\hat \lambda_1(x_j){^R}\hspace*{-2pt}\widehat{D}_-^\nu u_N(x_j)+\hat \lambda_2(x_j) u_N(x_j)=\hat f(x_j),\quad 1\le j\le N;\\
& u_N(-1)=u_-, \quad u_N(1)=u_+,\quad  0<\nu<\mu,\;\; \mu\in (1,2), 
\end{split}
\end{equation}
where $\hat\lambda_1=(1+x)^{\mu-\nu}\lambda_1,$ $\hat\lambda_2=(1+x)^{\mu}\lambda_2,$ and $\hat f=(1+x)^{\mu} f.$ 

Here, we  just focus on    the collocation system using  the new basis in \eqref{RLBirkform}, that is,  
\begin{equation}\label{RLBirkformB}
u_N(x)=u_- \widehat Q_0^{\mu}(x)+u_+\widehat Q_N^{\mu}(x)+\sum_{j=1}^{N-1} v_j \,\widehat Q_j^\mu(x),\quad \mu\in (1,2). 
\end{equation}
Then one  can write down  the B-COL system  in a fashion very similar to \eqref{COLsys1mu2} with only a change of basis. Correspondingly,  we denote the matrix of the linear system by $\bs A:=\bs I_{N-1}+\bs \Lambda_1 \widetilde  {\bs Q}^{(\nu)} +\bs \Lambda_2 \widehat {\bs Q}^{(\mu)}.$

We first show that the B-COL scheme enjoys  spectral accuracy (i.e., exponential convergence),  when the underlying solution is sufficiently smooth. For this purpose, we take 
\begin{equation}\label{exactu2}
u(x)=e^{-(1+x)}-\frac{1-x}{2}-e^{-2}\frac{1+x}{2},
\end{equation}
and $\lambda_1,\lambda_2$ to be  the same as in   \eqref{lam12}. In Figure \ref{spectraltest}, we plot discrete $L^2$-errors for various pairs of $(\mu,\nu)$ of the B-COL schemes for both  Caputo and Riemann-Liouville fractional boundary value problems (BVPs) 
\eqref{Bvp00} and \eqref{Bvp00A}  under the same setting. We observe the exponential decay  (i.e., $O(e^{-cN})$ for some $c>0$) of the errors.   
Both schemes take about $10$ iterations to  converge, while much more iterations are needed  and severe round-off errors  are induced   if one uses the standard  L-COL approach.   
\begin{figure}[h!]
  \centering
    \includegraphics[width=0.49\textwidth]{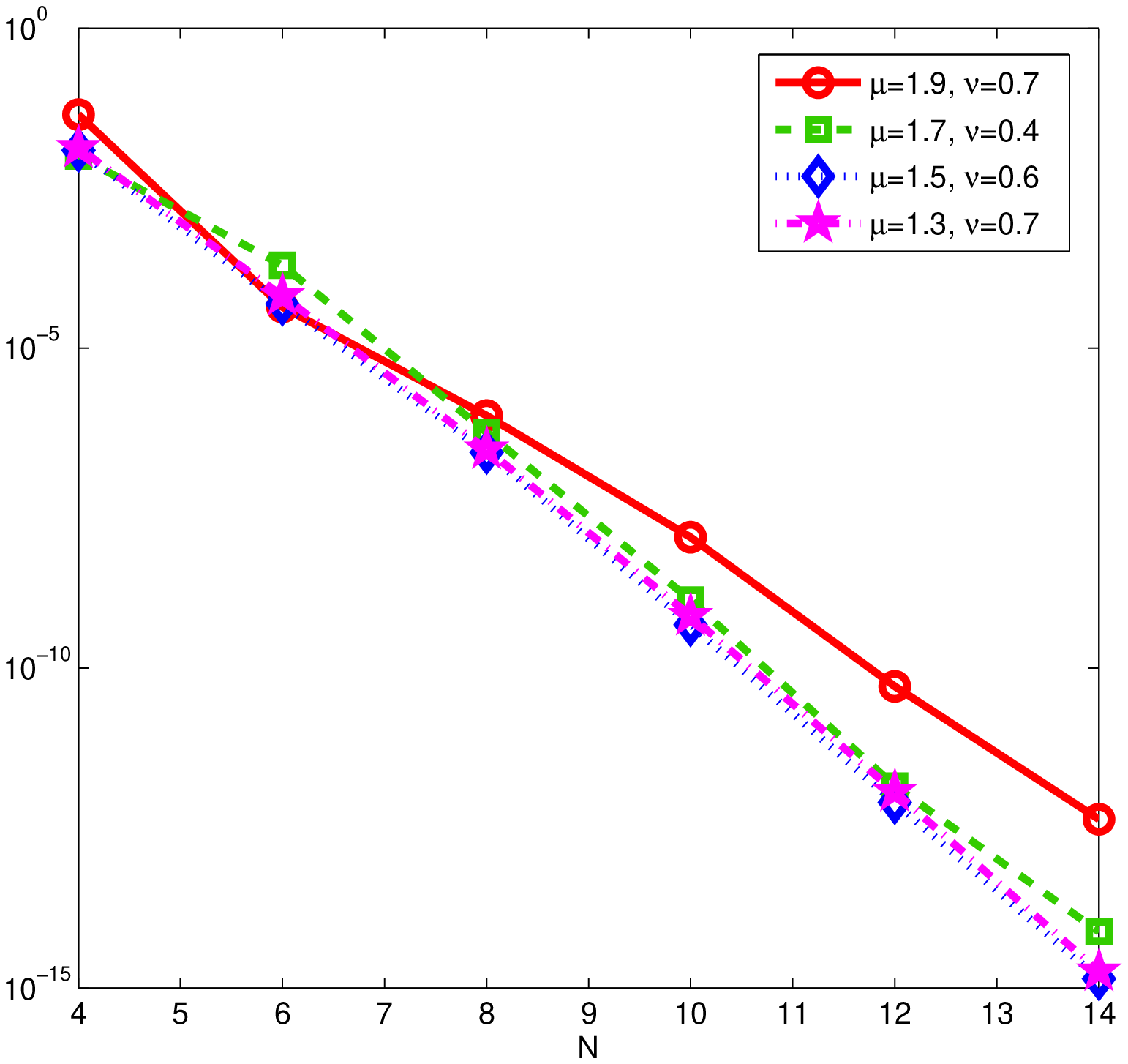}
     \includegraphics[width=0.49\textwidth]{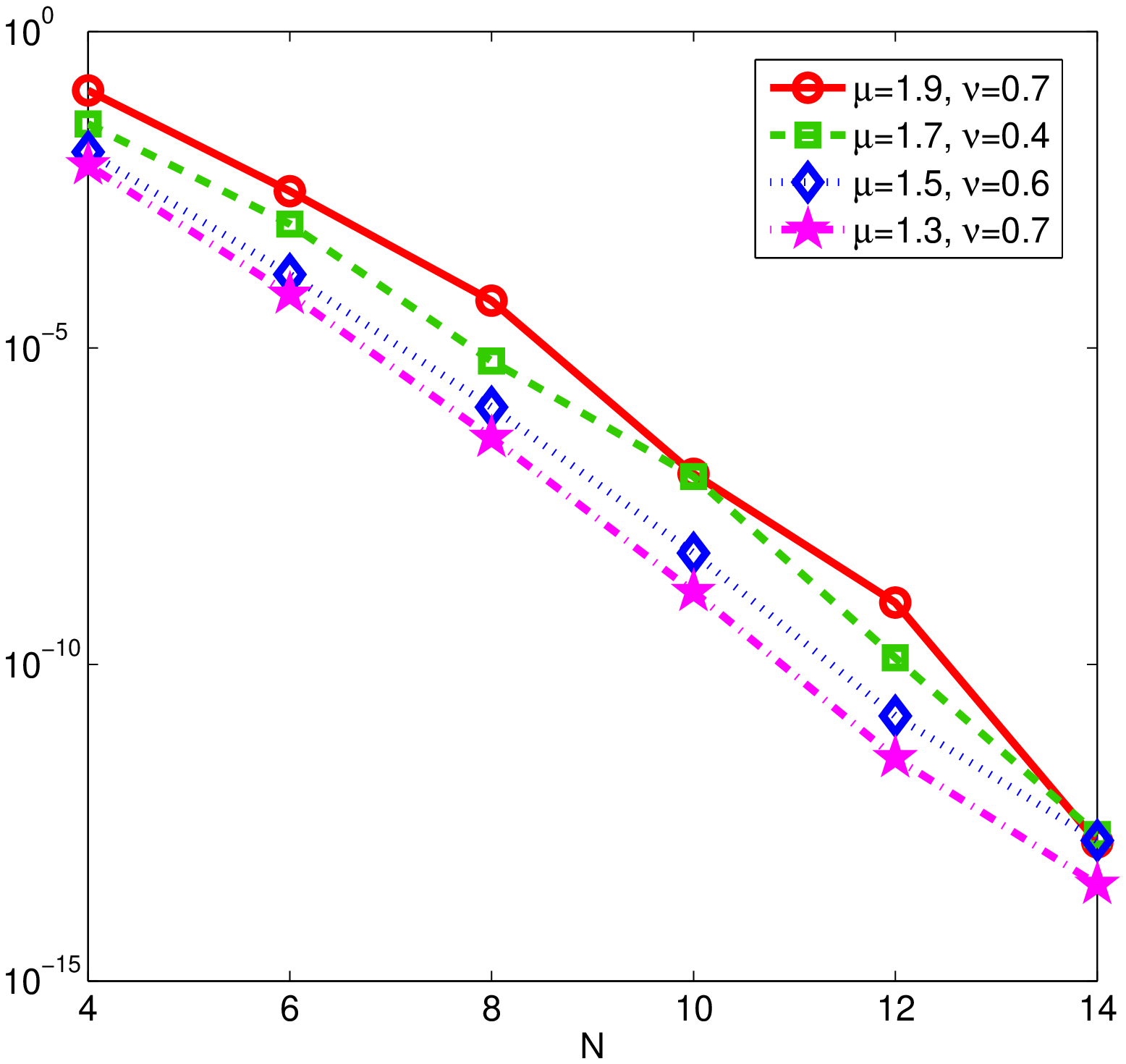}
   \caption{Errors against $N$   of  the B-COL schemes for Caputo fractional BVP   (\ref{Bvp00}) (left) and Riemann-Liouville fractional BVP (\ref{Bvp00A}) (right) with 
   various $\mu,\nu$, where $\lambda_1,\lambda_2$ are given in (\ref{lam12}), and the exact solution $u$ is given in (\ref{exactu2}).}
    \label{spectraltest} 
\end{figure}

We further test the new B-COL method on \eqref{Bvp00A} with smooth coefficients but large derivative: 
\begin{equation}\label{lam34}
\lambda_{1}(x)=1+e^{-1000x^2},\quad  \lambda_{2}(x)=1+e^{-1000(x+0.2)^2},
\end{equation} 
and with the exact solution having finite regularity in the usual Sobolev space:
 \begin{equation}\label{u34}
 u(x)=E_{\mu,1}\big(-(1+x)^\mu/2\big)+\frac{1-E_{\mu,1}(-2^{\mu-1})}{2}(1+x)-1.
 \end{equation}
 We tabulate in Table \ref{TbEigValueRLBVP3} the discrete $L^2$-errors, number of iterations  and the  second smallest and largest eigenvalues (in modulus).  Once again, the scheme converges within a few iterations even for very large $N.$
 In fact, as we observed from Figure \ref{uhatvsuNhat00} (right), the smallest eigenvalue of ${}^R\hspace*{-2pt}\widehat{\bs D}_{{\rm in}}^{(\mu)}$ in \eqref{submatixRL}  still mildly depends on $N.$ As a result, the condition number of $\bs A$ grows mildly with respect to $N.$  However, it is interesting to find that  the eigenvalues in modulus of $\bs A$ (denoted by $\{|\sigma_j|\}_{j=1}^{N-1}$ which are  arranged in ascending order)  are concentrated in the sense that 
 \begin{equation}\label{sigmacon}
 O(1)=|\sigma_2|\le |\sigma_j|\le |\sigma_{N-2}|=O(1), \quad 2\le j\le N-2.
 \end{equation}
 Thanks to this remarkable property,  the iterative solver for the modified Riemann-Liouville system is actually as fast as the previous  Caputo system where the coefficient matrix is well-conditioned.    
\begin{table}[!htbp]
\centering
\caption{Errors, number of iterations and concentration of eigenvalues of $\bs A$} \label{TbEigValueRLBVP3}
\begin{tabular}{|c|c|c|c|c|c|c|c|c|}
 \hline
  \multirow{2}{*}{$N$}&
    \multicolumn{4}{|c|}{$\mu=1.5,\;\;\nu=0.6$} &
      \multicolumn{4}{|c|}{$\mu=1.9,\;\;\nu=0.7$} \\
      \cline{2-9}
     &$|\sigma_2|$ & $|\sigma_{N-2}|$  & Iters &  Errors  &  $|\sigma_2|$ & $|\sigma_{N-2}|$  & Iters & Errors\\
\cline{1-9}
   8&0.7& 1.0&   7& 8.58e-03&0.4& 1.0&  6& 2.66e-03\\
  16&0.6& 1.1&  12& 2.03e-03&0.2& 1.7&  8& 3.69e-04\\
  32&0.6& 1.2&  12& 5.39e-04&0.2& 3.0&  8& 5.49e-05\\
  64&0.6& 1.3&  12& 1.31e-04&0.1& 5.0&  8&  7.46e-06\\
 128&0.6& 1.5&  12& 3.24e-05&0.1& 7.4&  8&  1.06e-06\\
 256&0.6& 1.6&  12& 8.08e-06&0.1& 9.9&  8&  1.53e-07\\
 512&0.6& 1.7&  13& 2.02e-06&0.1& 12.7& 8&  2.21e-08\\
1024&0.6& 1.7&  13& 5.04e-07&0.1& 21.1&  8&  3.69e-09\\
                 \hline
\end{tabular} 
\end{table}

\subsection{Concluding remarks}   In this paper, we provided an explicit and compact means for computing Caputo and modified Riemann-Liouville  F-PSDMs of any order, and introduced new fractional collocation schemes using fractional Birkhoff interpolation basis functions.   We showed that the new approaches significantly outperformed the standard collocation approximation using Lagrange interpolation basis. 

As a final remark, we point out two topics  worthy of future investigation along this line, which we wish to explore in  forthcoming  papers.  The first is to analyze the fractional Birkhoff interpolation errors and understand the approximability of  the new interpolation basis functions from theoretical perspective.  The second is to extend the idea and techniques in this paper to study the fractional collocation methods using the nodal basis $\big\{(1+x)^\mu h_j(x)/(1+x_j)^\mu\big\}$  (see \cite{zayernouri2014fractional}, i.e., the counterpart of Jacobi poly-fractonomials \cite{zayernouri2013fractional} and generalised Jacobi functions \cite{Wan.SZ14}).  

%

\vskip 10pt
\begin{appendix}
\section{Proof of Theorem \ref{JacobimatCa}} \label{AppendixB}
\renewcommand{\theequation}{A.\arabic{equation}}
\setcounter{equation}{0}

We expand $\{h_j'\}$
in terms of Legendre polynomials, and  look for $\{s_{lj}\}$ such that 
\begin{equation}\label{Dhj}
h_j'(x)= \frac 1 2\sum_{n=1}^N  (n+\alpha+\beta+1) t_{nj} P_{n-1}^{(\af+1,\bt+1)}(x)=\sum_{l=1}^N  s_{lj} P_{l-1}(x),  \quad 0\le j\le N,
\end{equation}
where $\{t_{nj}\}$ are given in \eqref{jacbasisfun} and we used  the derivative formula (cf. \cite{szeg75}): 
\begin{equation}\label{derivform}
D P_n^{(\af,\bt)}(x)=\frac 1 2 (n+\af+\bt+1) P_{n-1}^{(\af+1,\bt+1)}(x),\quad n\ge 1. 
\end{equation}
By \eqref{uluk}-\eqref{ucoefrela},   we find that 
\begin{equation}\label{connptbs}
 s_{lj}=\frac 1 2  \sum_{n=1}^N  (n+\alpha+\beta+1)  {}^{(\af+1,\bt+1)\!}C_{l-1,n-1}^{(0,0)}\, t_{nj}=\frac 1 2  \sum_{n=l-1}^N  (n+\alpha+\beta+1)  {}^{(\af+1,\bt+1)\!}C_{l-1,n-1}^{(0,0)}\, t_{nj},
\end{equation}
for $1\le l\le N$ and $0\le j\le N.$  In view of \eqref{Dhj}, we can use   the first formula in 
\eqref{specaseAB} to  derive 
\begin{equation}\label{lastaddedB}
\begin{split}
{}^C\hspace*{-3pt}{\bs D}^{(\mu)}_{ij}&:=\big({}^C\hspace*{-3pt}{D}^{\mu}_-h_j\big)(x_{i})= I^{1-\mu}_- h_j' (x_{i})=
\sum_{l=1}^{N}   s_{lj} I^{1-\mu}_- P_{l-1}(x_i)\\
&=(1+x_{i})^{1-\mu} \sum_{l=1}^N    \frac{(l-1)!}{\Gamma(l+1-\mu)}\, s_{lj}\,  P_{l-1}^{(\mu-1,1-\mu)}(x_{i})\,.
\end{split}
\end{equation}
This ends the derivation of \eqref{DmuJGL}-\eqref{slj2}. 


We now derive  \eqref{tnjnj} for the LGL case.  Using the orthogonality of Legendre polynomials, integration by parts and the exactness of LGL quadrature (cf. \eqref{newquad}), we obtain from \eqref{Dhj} that  
\begin{equation}\label{ttildeformu}
\begin{split}
 s_{lj}&=\frac 1{\gamma_{l-1}}\int_{-1}^1 h_j'(x)P_{l-1}(x)\,dx
= \frac 1{\gamma_{l-1}}\bigg\{h_j(x)P_{l-1}(x)\big|_{-1}^1-\sum_{i=0}^N h_j(x_i)P_{l-1}'(x_i) \omega_i\bigg\}\\
&=\frac 1{\gamma_{l-1}}\big\{ \delta_{jN}+(-1)^l \delta_{j0}- \omega_j P_{l-1}'(x_j)\big\},
\end{split}
\end{equation}
where we used the properties: $h_j(x_i)=\delta_{ij}$  and $P_{l-1}(\pm 1)=(\pm 1)^{l-1}.$ 

\section{Proof of Theorem  \ref{newJacobimat}} \label{AppendixC}
\renewcommand{\theequation}{B.\arabic{equation}}
\setcounter{equation}{0}

Since $\hbar_j\in {\mathcal P}_{N-k},$ we can write 
\begin{equation}\label{ExpandharByLeg}
\hbar_{j}(x)=\displaystyle\sum^{N-k}_{n=0} \xi_{nj}P_{n}^{(\af,\bt)}(x)
=\sum^{N-k}_{l=0}\breve \xi_{lj} P^{(\mu-k,k-\mu)}_{l}(x),\;\;\; 1\le j\le N+1-k.
\end{equation}
As before, if one can work out $\{\xi_{nj}\},$  then by \eqref{uluk}-\eqref{ucoefrela}, 
\begin{align}\label{xilj}
\breve \xi_{lj}=\sum_{n=l}^{N-k} {}^{(\af,\bt)\!}C_{ln}^{(\mu-k,k-\mu)} \xi_{nj}.
\end{align}
As to be shown later, inserting  \eqref{ExpandharByLeg} into \eqref{newidentityA}, we can derive from 
\eqref{specaseAB2} with $\rho=k-\mu$ the desired formulas.  Thus, it remains to find  $\{\xi_{nj}\}$ in 
\eqref{ExpandharByLeg}.  We proceed separately for two cases. 

\vskip 4pt
(i) For $\mu\in (0,1),$ we obtain from  the orthogonality \eqref{jcbiorth}, the exactness of JGL quadrature  \eqref{newquad}, and the interpolating condition \eqref{hbari} that 
\begin{equation}\label{hjaj1}
\begin{split}
\xi_{nj}&=\frac{1}{\gamma_n^{(\af,\bt)}}\int^{1}_{-1} \hbar_{j}(x)P_{n}^{(\af,\bt)}(x)\omega^{(\af,\bt)}(x)dx=\frac{1}{\gamma_n^{(\af,\bt)}}\sum_{i=0}^N \hbar_j(x_i) P_{n}^{(\af,\bt)}(x_i)\omega_i \\
&=\frac{1}{\gamma_n^{(\af,\bt)}} \big\{\hbar_{j}(-1)P_{n}^{(\af,\bt)}(-1)\omega_0
+P_{n}^{(\af,\bt)}(x_j)\omega_j\big\},\quad  1\le j\le N.
\end{split}
\end{equation}
Now, we evaluate $\hbar_j(-1).$  Since $\{\hbar_j\}$ are associated with the interpolating points $\{x_j\}_{j=1}^N,$ which are zeros of $(1-x)DP_N^{(\af,\bt)}(x),$ we have the representation: 
\begin{equation}\label{InterPolyhbarJGL1}
\hbar_{j}(x)=\frac{(1-x)D P^{(\alpha,\beta)}_{N}(x)}
{(x-x_j) D\big\{ (1-x)D P^{(\alpha,\beta)}_{N}(x)\big\}\big|_{x=x_j}},
\quad  1\leq j\leq N.\end{equation}
Recall the Sturm-Liouville equation of Jacobi polynomials (cf. \cite[(4.2.1)]{szeg75}):
\begin{equation}\label{StLiuEQinJa}
-(1-x^2)D^2 P^{(\alpha,\beta)}_{N}(x)=\big\{\beta-\alpha-(\alpha+\beta+2)x\big\}D
P^{(\alpha,\beta)}_{N}(x)+\lambda^{(\alpha,\beta)}_N P^{(\alpha,\beta)}_{N}(x),
\end{equation}
where $\lambda_N^{(\af,\bt)}=N(N+\alpha+\beta+1).$ It follows from \eqref{StLiuEQinJa} that 
\begin{equation}\label{newvalues}
\begin{split}
& 2(\beta+1) DP_N^{(\alpha,\beta)}(-1)=-\lambda^{(\alpha,\beta)}_N  P_N^{(\alpha,\beta)}(-1), \;\;\;  2(\alpha+1) DP_N^{(\alpha,\beta)}(1)=\lambda^{(\alpha,\beta)}_N  P_N^{(\alpha,\beta)}(1), \\
&  -(1-x^2_j)D^2 P^{(\alpha,\beta)}_{N}(x_j)=\lambda^{(\alpha,\beta)}_N P^{(\alpha,\beta)}_{N}(x_j),\quad 1\le j\le N-1.
\end{split}
\end{equation}
 Using the property:   $(1-x_j)DP_N^{(\af,\bt)}(x_j)=0$ and \eqref{newvalues}, we compute from  \eqref{InterPolyhbarJGL1} that 
 \begin{equation}\label{hfu1}
 \hbar_j(-1)=-\frac {c_j} {\beta+1} \frac{P_N^{(\alpha,\beta)}(-1)}{P_N^{(\alpha,\beta)}(x_j)}, \;\;\; 1\le j\le N,
 \end{equation}
 where $c_j=1$ for $1\le j\le N-1,$ and $c_N=\alpha+1.$ Substituting \eqref{hfu1} into \eqref{hjaj1} yields  \eqref{hjaj10}.

Inserting  \eqref{ExpandharByLeg} into \eqref{newidentityA}, we derive from 
\eqref{specaseAB2} with $\rho=1-\mu$ that 
\begin{equation}\label{dQj22}
DQ_j^\mu(x)=\frac{1}{(1+x_j)^{1-\mu}}\sum_{l=0}^{N-1} \frac{\Gamma(l-\mu+2)}{l!}\, \breve\xi_{lj}\,  P_l(x), \quad 1\le j\le N.
\end{equation}
In view of $Q_j^\mu(-1)=0,$ a direct integration of \eqref{dQj22} leads to \eqref{dQj}.

\vskip 4pt 
(ii)  For $\mu\in (1,2),$ \eqref{hjaj1} reads 
\begin{equation}\label{hjaj12}
\begin{split}
\xi_{nj}
&=\frac{1}{\gamma_n^{(\af,\bt)}} \big\{\hbar_{j}(-1)P_{n}^{(\af,\bt)}(-1)\omega_0
+ \hbar_{j}(1)P_{n}^{(\af,\bt)}(1)\omega_N
+P_{n}^{(\af,\bt)}(x_j)\omega_j\big\},\quad  1\le j\le N-1.
\end{split}
\end{equation}
We need to evaluate $\hbar_j(\pm 1).$ Note that in this case, $\{\hbar_j\}$ are associated with the interior JGL points  $\{x_j\}_{j=1}^{N-1},$ which are zeros of $DP_N^{(\af,\bt)}(x),$ so we have 
\begin{equation}\label{Interhj}
\hbar_{j}(x)=\frac{DP^{(\alpha,\beta)}_{N}(x)}
{(x-x_j)D^2P^{(\alpha,\beta)}_{N}(x_j)},
\quad 1\leq j\leq N-1.
\end{equation}
Thus using \eqref{StLiuEQinJa}-\eqref{newvalues} leads to 
\begin{equation}\label{newhjf1}
\hbar_j(-1)=-\frac{1-x_j}{2(\bt+1)}\frac{P_N^{(\af,\bt)}(-1)}{P_N^{(\af,\bt)}(x_j)},\quad
\hbar_j(1)=-\frac{1+x_j}{2(\af+1)}\frac{P_N^{(\af,\bt)}(1)}{P_N^{(\af,\bt)}(x_j)}. 
\end{equation}
Substituting \eqref{newhjf1} into \eqref{hjaj12} yields  \eqref{hjaj102}.

Similar to case (i), inserting  \eqref{ExpandharByLeg} into \eqref{newidentityA}, we derive from 
\eqref{specaseAB2} with $\rho=2-\mu$ that 
\begin{equation}\label{dQj2pf}
D^2Q^\mu_{j}(x)= \frac{1}{(1+x_j)^{2-\mu}} \sum^{N-2}_{l=0}\frac{\Gamma(l-\mu+3)}{l!}
\; \breve \xi_{lj}\, P_{l}(x).
\end{equation}
Solving this equation with the boundary conditions:  $Q_j^\mu(\pm 1)=0,$ we obtain $\Phi_l$ in \eqref{Phisbas}   and the desired formula \eqref{dQj2}.

\section{Proof of Lemma  \ref{zerolmm}} \label{AppendixC0}
\renewcommand{\theequation}{C.\arabic{equation}}
\setcounter{equation}{0}

We carry out the proof by directly verifying that $u(x)$ in \eqref{soluform} is the desired polynomial solution.  It is evident that for any $f\in {}_0{\mathcal P}_N,$ we can  write 
\begin{equation}\label{newpfA}
f(x)=\sum_{n=0}^{N-1} \hat f_n\,(1+x) P_n^{(\mu,1-\mu)}(x), 
\end{equation}
where the coefficients $\{\hat f_n\}$ are uniquely determined.   Using  \eqref{newbatemanam} with 
$\rho=\mu, \alpha=\mu$ and $\beta=1-\mu,$  leads to 
\begin{equation}\label{newpfB}
u(x)=I_-^\mu \big\{(1+x)^{-\mu} f(x)\big\}= \sum_{n=0}^{N-1} \frac{\Gamma(n+2-\mu)}{(n+1)!} \hat f_n\,(1+x) P_n^{(0,1)}(x), 
\end{equation}
which implies $u\in {}_0{\mathcal P}_N.$ 
Recall that  ${^R}\hspace*{-2pt} {\widehat D}^{\mu}_-=(1+x)^{\mu} {^R}\hspace*{-2pt}D_-^\mu.$ Thus, acting  ${^R}\hspace*{-2pt} {\widehat D}^{\mu}_-$ on both sides of  \eqref{newpfB}, we obtain from  \eqref{newbatemanam3s} and \eqref{newpfA} immediately that 
\begin{equation}\label{newpfC}
\begin{split}
{^R}\hspace*{-2pt} {\widehat D}^{\mu}_-u(x) &=\sum_{n=0}^{N-1} \frac{\Gamma(n+2-\mu)}{(n+1)!} \hat f_n\, {^R}\hspace*{-2pt} {\widehat D}^{\mu}_-\big\{(1+x) P_n^{(0,1)}(x)\big\}\\&=\sum_{n=0}^{N-1} \hat f_n\,(1+x) P_n^{(\mu,1-\mu)}(x)=f(x). 
\end{split}
\end{equation}
Therefore, $u(x)$ in \eqref{soluform}  verifies \eqref{newsolu}. The uniqueness follows from  ${^R}\hspace*{-2pt} {\widehat D}^{\mu}_-u(x)=0$ implying 
$u(x)=0.$

We now turn to \eqref{Rrhoeqn}.  The above verification shows that if $f(-1)=0,$ i.e.,  
${^R}\hspace*{-2pt} {\widehat D}^{\rho}_-u(-1)=0,$ then $u(-1)=0.$ Hence, it suffices to show if 
$u(-1)=0,$ then ${^R}\hspace*{-2pt} {\widehat D}^{\rho}_-u(-1)=0.$  For this purpose, we expand 
\begin{equation}\label{newpfD}
u(x)=\sum_{n=0}^{N-1} \hat u_n\, (1+x)P_n^{(0,1)}(x),
\end{equation}
where $\{\hat u_n\}$ can be uniquely determined. Like the derivation of \eqref{newpfC}, 
 acting  ${^R}\hspace*{-2pt} {\widehat D}^{\mu}_-$ and using   \eqref{newbatemanam3s}, we obtain
\begin{equation}\label{newpfE}
{^R}\hspace*{-2pt} {\widehat D}^{\mu}_-u(x) =\sum_{n=0}^{N-1}  \frac {(n+1)!} {\Gamma(n+2-\mu)} \,\hat u_n\,(1+x) P_n^{(\mu,1-\mu)}(x),
\end{equation}
which implies ${^R}\hspace*{-2pt} {\widehat D}^{\rho}_-u(-1)=0.$

\section{Proof of Theorem  \ref{invRLmu01} } \label{Appendixnew}
\renewcommand{\theequation}{D.\arabic{equation}}
\setcounter{equation}{0}

We intend to use the compact identity deduced  from \eqref{newbatemanam3s}, that is, 
\begin{equation}\label{RLmu01}
{}^R\hspace*{-2pt}\widehat D_-^\mu \big\{ P_n(x)\big\}
=\frac{n!} {\Gamma(n-\mu+1)} P_n^{(\mu,-\mu)}(x),\quad n\ge 0,\quad \mu\in(0,1).  
\end{equation}
 This inspires us to expand $\{h_j\}$ (resp. $\widehat Q_j^\mu$) in terms of $\{P_l^{(\mu,-\mu)}\}$  
 (resp. $\{P_l\}$). Following \eqref{matrxA}-\eqref{exastD}, we have 
 \begin{equation}\label{exastD00}
h_j(x)=\sum_{l=0}^N \hat t_{lj} P_l^{(\mu,-\mu)}(x),\quad \hat t_{lj}=
\sum_{n=l}^N {}^{(\af,\bt)\!}C_{ln}^{(\mu,-\mu)} t_{nj},
\end{equation}
and 
\begin{equation}\label{QjLeg}
\widehat Q_j^\mu(x)=\sum_{l=0}^N \hat q_{lj} P_l(x),\quad 0\le j\le N. 
\end{equation}
Inserting \eqref{exastD00}-\eqref{QjLeg} into \eqref{identiyRLj}, we obtain from \eqref{RLmu01} immediately that for $ 0\le l\le N,$
\begin{equation}\label{qljformu}
\hat q_{lj}=\frac{\Gamma(l-\mu+1)} {l!}\, \hat t_{lj}, \quad 1\le j\le N,\quad   \hat q_{l0}=\frac{\Gamma(l-\mu+1)} {l!} \xi\, \hat t_{l0}.
\end{equation}
Thus, it remains to determine the constant $\xi$.  Setting  $\widetilde  Q_0^\mu(x)=\widehat Q_0^\mu(x)-1,$  we have $\widetilde  Q_0^\mu(-1)=0.$ Using \eqref{Rrhoeqn},   the formula \eqref{rlpower}, and definition \eqref{eqnasd}, we obtain from \eqref{identiyRLj} that 
\begin{equation}\label{identiyRL0J}
0=\big({}^R\hspace*{-2pt}\widehat{ D}_-^{\mu}\,\widetilde Q_0^\mu\big)(-1)=\xi-
\big({}^R\hspace*{-2pt}\widehat{ D}_-^{\mu}\,1\big)\big|_{x=1}=\xi-\frac{1}{\Gamma(1-\mu)},\;\; {\rm so}\;\; \xi=\frac{1}{\Gamma(1-\mu)}. 
\end{equation}
This ends the proof.

\section{Proof of Lemma  \ref{Jcbiexplmm} } \label{AppendixD}
\renewcommand{\theequation}{E.\arabic{equation}}
\setcounter{equation}{0}
 
  We first derive the coefficients in \eqref{newconstA}-\eqref{newconstBj}. 
By  the orthogonality \eqref{jcbiorth}, the exactness of JGL quadrature  \eqref{newquad}, and the interpolating condition \eqref{hathj}, we have 
  \begin{equation}\label{hjaj1A0}
\begin{split}
\varrho_{nj}&=\frac{1}{\gamma_n^{(\af,\bt)}}\int^{1}_{-1} \hat h_{j}(x)P_{n}^{(\af,\bt)}(x)\omega^{(\af,\bt)}(x)dx=\frac{1}{\gamma_n^{(\af,\bt)}}\sum_{i=0}^N \hat h_j(x_i) P_{n}^{(\af,\bt)}(x_i)\omega_i \\
&=\frac{1}{\gamma_n^{(\af,\bt)}} \big\{P_{n}^{(\af,\bt)}(x_j)\omega_j+\hat h_{j}(1)P_{n}^{(\af,\bt)}(1)\omega_N\big\},\quad  0\le n,j\le N-1.
\end{split}
\end{equation}
 Since $\{\hat h_j\}$ are associated with the JGL points $\{x_j\}_{j=0}^{N-1},$ which are zeros of $(1+x)DP_N^{(\af,\bt)}(x),$ we have the representation: 
\begin{equation}\label{InterGL1}
\hat h_{j}(x)=\frac{(1+x)D P^{(\alpha,\beta)}_{N}(x)}
{(x-x_j) D\big\{ (1+x)D P^{(\alpha,\beta)}_{N}(x)\big\}\big|_{x=x_j}},
\quad  0\leq j\leq N-1.
\end{equation}
A direct calculation  using \eqref{newvalues} leads to 
\begin{equation}\label{asphas}
\hat h_0(1)=-\frac{\beta+1}{\alpha+1} \frac{P_N^{(\af,\bt)}(1)} {P_N^{(\af,\bt)}(-1)},\quad  \hat h_j(1)
=-\frac{1}{\alpha+1} \frac{P_N^{(\af,\bt)}(1)} {P_N^{(\af,\bt)}(x_j)},\;\;\; 1\le j\le N-1.  
\end{equation} 
  Thus, we obtain \eqref{newconstA}-\eqref{newconstBj} by inserting them into \eqref{hjaj1A0}.

 Thanks to 
  \begin{equation}\label{hathjexp22}
  \hat h_j(x)=\sum_{n=0}^{N-1}\varrho_{nj}\, P_n^{(\af,\bt)}(x)=\sum_{l=0}^{N-1}\tilde \varrho_{lj}\, P_l^{(\mu,1-\mu)}(x),\quad 0\le j\le N-1, 
  \end{equation} 
  we solve the connection problem and obtain from \eqref{ucoefrela}-\eqref{constAB} the formula  
  \eqref{connectjj}.

It remains to derive \eqref{backformula}.  Applying the three-term recurrence relation \eqref{ThreeTermExpress} to the last expansion in \eqref{hathjexp}, we obtain the connection  
\begin{equation}\label{bsTsj00}
 \bs T \bs {\hat \varrho_j}=\bs {\tilde \varrho_j}, \;\;\; \bs {\hat \varrho_j}=(\hat \varrho_{0j},\cdots,\hat \varrho_{N-1,j} )^t,\;\; \bs {\tilde \varrho_j}=(\tilde \varrho_{0j},\cdots,\tilde \varrho_{N-1,j})^t,
\end{equation}
where 
 $\bs T$ is an upper triangular matrix with only nonzero entries on diagonal and two upper diagonals: 
\begin{equation}\label{transeqn00}
\bs T_{00}=1,\;\; \bs T_{ii}=a_{i};\;\;\; \bs T_{i,i+1}=b_{i}+1;\quad    \bs T_{i,i+2}=c_i.
\end{equation}
Solving the linear system by backward substitution leads to  \eqref{backformula}.

\section{Proof of Theorem  \ref{invRLmu12}} \label{AppendixE}
\renewcommand{\theequation}{F.\arabic{equation}}
\setcounter{equation}{0}
 
We first use  Lemma \ref{zerolmm} to solve   \eqref{identiyRL02}-\eqref{identiyRLN2} and find the expressions of the constants therein.  
 It's more convenient to reformulate \eqref{identiyRL02} as: find $\widehat Q_0^\mu(x)=\breve Q_0^\mu(x)+1$ such that  
\begin{equation}\label{eqnAs}
\big({}^R\hspace*{-2pt}\widehat{ D}_-^{\mu}\,\breve Q_0^\mu\big)(x)=
\Big(\tau_0 -\frac 1{\Gamma(1-\mu)} \Big)
  (1+x) \hat h_0(x) +\frac 1{\Gamma(1-\mu)} \big(\hat h_0(x)-1), \quad    \breve Q_0^\mu(1)=-1,
\end{equation}
where we used \eqref{rlpower}, \eqref{fpsdmod} and \eqref{Rrhoeqn} to derive 
\begin{equation}\label{QN1A}
{}^R\hspace*{-2pt}\widehat{ D}_-^{\mu}1=\frac{1}{\Gamma(1-\mu)},\quad \kappa_0=\tau_0-\frac 1 {\Gamma(1-\mu)}.
\end{equation}
Using Lemma \ref{zerolmm} and  \eqref{intformu}, we obtain 
\begin{equation}\label{Q0solu}
\breve Q_0^\mu(x)= \Big(\tau_0 -\frac 1{\Gamma(1-\mu)} \Big)
I_-^\mu \big\{(1+x)^{1-\mu} \hat h_0(x)\big\}+ \frac 1{\Gamma(1-\mu)} I_-^\mu \big\{(1+x)^{-\mu} 
(\hat h_0(x)-1)\big\}. 
\end{equation} 
As $\breve Q_0^\mu(1)=-1$, we have 
 \begin{align}\label{tau0value}
\Gamma(1-\mu)  \,\tau_0=1-\frac {I_-^\mu \big\{(1+x)^{-\mu}  (\hat h_0(x)-1)\big\}\big|_{x=1}+\Gamma(1-\mu)} {I_-^\mu \big\{(1+x)^{1-\mu} \hat h_0(x)\big\}\big|_{x=1}}\,.
 \end{align}
%
Following the same argument, we derive 
\begin{align}
 &\widehat Q_N^\mu(x)= \tau_N\, I_-^\mu \big\{(1+x)^{1-\mu} \hat h_0(x)\big\}, \quad \tau_N=\frac 1 {I_-^\mu \big\{(1+x)^{1-\mu} \hat h_0(x)\big\}\big|_{x=1}},\label{QNsolu}
 \end{align}
 and for $1\le j\le N-1,$
 \begin{align}
 &\widehat Q_j^\mu(x)= \frac {1}{x_j+\tau_j}\Big(\tau_j I_-^\mu \big\{(1+x)^{-\mu} \hat h_j(x)\big\} + 
 I_-^\mu \big\{(1+x)^{-\mu}x \hat h_j(x)\big\}\Big), \label{Qjsolu}\\
 &\tau_j=-\frac {I_-^\mu \big\{(1+x)^{-\mu} x \hat h_j(x)\big\}\big|_{x=1}}  
{ I_-^\mu \big\{(1+x)^{-\mu} \hat h_j(x)\big\}\big|_{x=1}}. \label{taujsolu}
 \end{align}

 We now evaluate fractional integrals of $\hat h_j$. 
 Using the last two expansions with $j=0$ in \eqref{hathjexp}, and the identity   \eqref{newbatemanam} with 
$\rho=\mu, \alpha=\mu$ and $\beta=1-\mu$, we obtain
\begin{equation}\label{newpfBJ}
\begin{split}
& I_-^\mu \big\{(1+x)^{1-\mu} \hat h_0(x)\big\}= \sum_{l=0}^{N-1} \frac{\Gamma(l+2-\mu)}{(l+1)!} 
\tilde  \varrho_{l0}\,(1+x) P_l^{(0,1)}(x), \\
& I_-^\mu \big\{(1+x)^{-\mu} (\hat h_0(x)-1)\big\}= \sum_{l=0}^{N-2} \frac{\Gamma(l+2-\mu)}{(l+1)!} 
\hat  \varrho_{l+1,0}\,(1+x)P_l^{(0,1)}(x). 
\end{split}
\end{equation}
Noting that $P_n^{(0,1)}(1)=1$ (cf.  \cite{szeg75}),  we obtain from  \eqref{tau0value} and \eqref{newpfBJ} the value of $\tau_0$ in 
\eqref{Q0solunewm}, and the expression of $\widehat Q_0^\mu(x)$ follows from \eqref{Q0solu} immediately.

Similarly, we obtain from \eqref{QNsolu} and \eqref{newpfBJ}  the expression of $\widehat Q_N^\mu(x)$ in \eqref{QNformula}.

We now turn to  $\widehat Q_j^\mu(x)$ with $1\le j\le N-1.$ 
Once again,  using  \eqref{newbatemanam} (with 
$\rho=\mu, \alpha=\mu$ and $\beta=1-\mu$)  and \eqref{hathjexp},  leads to 
\begin{equation}\label{newpfBJj}
\begin{split}
& I_-^\mu \big\{(1+x)^{-\mu} \hat h_j(x)\big\}= (1+x) \sum_{l=0}^{N-2} \frac{\Gamma(l+2-\mu)}{(l+1)!} 
\hat  \varrho_{l+1,j}\, P_l^{(0,1)}(x),  \\
&I_-^\mu \big\{(1+x)^{-\mu} \hat h_j(x)\big\}\big|_{x=1}= 2\sum_{l=0}^{N-2} \frac{\Gamma(l+2-\mu)}{(l+1)!} 
\hat  \varrho_{l+1,j}\,,
\end{split}
\end{equation}
where we used $P_l^{(0,1)}(1)=1.$
Moreover, by  \eqref{newbatemanam},  \eqref{hathjexp} and  \eqref{ThreeTermExpress}-\eqref{3termcoef}, 
\begin{equation}\label{newpfBJj2}
\begin{split}
& I_-^\mu \big\{(1+x)^{-\mu} x\hat h_j(x)\big\}= \sum_{l=0}^{N-2} 
\hat  \varrho_{l+1,j}\, I_-^\mu \big\{(1+x)^{1-\mu} x P_l^{(\mu,1-\mu)}(x)\big\}=(1+x)\times  \\
& \;\; \sum_{l=0}^{N-2} \frac{\Gamma(l+2-\mu)}{(l+1)!} 
\hat  \varrho_{l+1,j}\,\bigg\{\frac{l+2-\mu}{l+2} a_{l+1} P_{l+1}^{(0,1)}+b_l P_{l}^{(0,1)}+  \frac{l+1}{l+1-\mu} c_{l-1}P_{l-1}^{(0,1)}\bigg\}(x),
\end{split}
\end{equation}
where $c_{-1}=0.$ Using the property: $P_{l-1}^{(0,1)}(1)=1$ and \eqref{3termcoef}, we find from a direct calculation and \eqref{newpfBJj} that 
\begin{equation}\label{value1}
\begin{split}
I_-^\mu \big\{(1+x)^{-\mu} x\hat h_j(x)\big\}\big|_{x=1}&= 2(1-\mu)\Gamma(2-\mu) \hat \varrho_{1j}+ 2 \sum_{l=1}^{N-2} \frac{\Gamma(l+2-\mu)}{(l+1)!} 
\hat  \varrho_{l+1,j}\\
&=I_-^\mu \big\{(1+x)^{-\mu} \hat h_j(x)\big\}\big|_{x=1}- 2\mu\,\Gamma(2-\mu) \hat \varrho_{1j}.
\end{split}
\end{equation}
Inserting \eqref{newpfBJj}-\eqref{value1} into \eqref{Qjsolu}-\eqref{taujsolu}, we derive the forumlas \eqref{Qjsoluform}-\eqref{tauj}.

\end{appendix}

\bibliography{RefFractional,ref,Refcol}

 \end{document}